\documentclass[11pt]{article}
\usepackage[margin=1in]{geometry}
\usepackage{amsmath,amsfonts,amssymb}
\usepackage{amsthm}
\usepackage{algorithm}
\usepackage[noend]{algpseudocode}
\usepackage{graphicx}
\usepackage{tikz}
\usetikzlibrary{arrows.meta,positioning,calc,decorations.pathmorphing}
\usepackage{array}
\usepackage{xcolor}
\definecolor{niceblue}{RGB}{40,90,200}
\usepackage[normalem]{ulem}

\definecolor{explaincolor}{RGB}{0,120,80}

\usepackage{placeins}
\usepackage[hidelinks]{hyperref}

\theoremstyle{plain}
\newtheorem{theorem}{Theorem}
\newtheorem{lemma}{Lemma}
\newtheorem{corollary}{Corollary}
\newtheorem{proposition}{Proposition}
\theoremstyle{definition}
\newtheorem{definition}{Definition}
\newtheorem{assumption}{Assumption}
\theoremstyle{remark}
\newtheorem{remark}{Remark}

\newcommand{\prox}{\mathrm{prox}}
\newcommand{\argmin}{\operatorname*{arg\,min}}

\newcommand{\R}{\mathbb{R}}
\newcommand{\norm}[1]{\left\|#1\right\|}
\newcommand{\abs}[1]{\left|#1\right|}
\DeclareMathOperator{\dist}{dist}
\DeclareMathOperator{\dom}{dom}
\newcolumntype{L}[1]{>{\raggedright\arraybackslash}p{#1}}
\makeatletter
\providecommand{\theHALG@line}{}
\renewcommand{\theHALG@line}{\thealgorithm.\arabic{ALG@line}}
\makeatother

\title{An Adaptive Proximal Framework for Weakly Convex Optimization with Unknown Parameter\\[4pt]
\large Deterministic and Heavy-Tailed Stochastic Guarantees}
\author{Miaolan Xie\thanks{Edwardson School of Industrial Engineering, Purdue University, West Lafayette, IN, USA. \texttt{miaolanx@purdue.edu}, \url{https://miaolan.github.io/}}}
\date{}

\begin{document}
\maketitle

\begin{abstract}
Many nonsmooth, nonconvex objectives in learning and signal recovery are $\rho$-weakly convex. We minimize such a function in deterministic and stochastic settings when the weak-convexity parameter $\rho$ is unknown. The objective is not required to be globally Lipschitz continuous or smooth. We propose the Adaptive Prox-Guided Scheme (APS), a single-trial adaptive proximal framework that adapts the proximal parameter online and bidirectionally through a descent test, allowing it to exploit favorable local structure. 

The framework covers adaptive proximal-point, prox-linear, and proximal-gradient methods from a model-based viewpoint. In the deterministic proximal setting, APS obtains an $\mathcal{O}(\varepsilon^{-2})$ iteration complexity for producing an $\varepsilon$-subgradient stationary point. In the stochastic setting, APS achieves a high-probability $\mathcal O(\varepsilon^{-2})$ iteration bound for driving the Moreau-envelope gradient below $\varepsilon$. This result holds under deliberately weak oracle assumptions: the function-difference estimates may be biased and heavy-tailed, and the stochastic candidate oracle need only be sufficiently accurate with constant probability when the proximal parameter is small, and can be arbitrary otherwise.
\end{abstract}


\bigskip

\section{Introduction}\label{sec:intro}
Nonsmooth nonconvex objectives arise throughout modern machine learning,
statistics, and signal processing. Weak convexity provides a broad structural
class for such problems while retaining enough regularity to support finite-time
stationarity guarantees. 
A proper, lower semicontinuous function $f:\mathbb R^d\to\mathbb R{\cup\{+\infty\}}$ is
$\rho$-weakly convex if
$x\mapsto f(x)+\frac{\rho}{2}\|x\|^2$ is convex.
This class contains all (possibly nonsmooth) convex functions and all smooth functions with
Lipschitz continuous gradients, and it includes composite objectives
$f(x)=h(c(x))$ with $h$ convex and Lipschitz and $c$ smooth with Lipschitz
Jacobian. Examples include phase retrieval~\cite{DuchiRuan2018,EldarMendelson2014,DuchiRuanPhase}, robust PCA and low-rank matrix recovery~{\cite{CandesRPCA2011,ChandrasekaranRankSparsity2011,CharisopoulosLowRank2021}}, conditional value-at-risk minimization~{\cite{RockafellarUryasev2000,RockafellarUryasev2002}}, and sparse regression~{\cite{BoehmWrightVariableSmoothing}}.
We study
\[
    \min_{x\in\mathbb R^d} f(x),
\]
where \(f\) is weakly convex and bounded below, in both deterministic and
stochastic settings. 

The central difficulty addressed in this paper is that proximal-based methods for weakly convex optimization typically require a safe choice of the proximal parameter, which in turn depends on the weak-convexity constant \mbox{$\rho$}. 
In many applications, \mbox{$\rho$} is unknown.
Because this proximal parameter controls the step size, its value directly governs how fast the method converges, which makes it central for algorithm efficiency. The standard remedy is to tune it by hand, but tuning is fragile: chosen too large, the steps become unstable and the iterates may diverge; chosen too small, the steps are over-regularized and progress is needlessly slow. Estimating or upper-bounding $\rho$ is a natural first attempt~\cite{DuchiRuanPhase}, but any global estimate is often overly conservative relative to the local curvature that can be exploited for faster progress. A monotonically decreasing schedule, on the other hand, can never increase the parameter again to exploit the benign regions encountered along the way. 
Moreover, existing parameter-free algorithms typically require additional assumptions, such as Lipschitz continuity or smoothness.

We introduce the Adaptive Prox-Guided Scheme (APS), a general proximal-based framework that addresses all of these issues simultaneously: it adapts the proximal parameter \emph{dynamically, in response to the algorithm's observed progress} -- raising it where the local geometry is favorable and lowering it only when necessary. It requires neither knowledge of $\rho$ nor any Lipschitz continuity or smoothness assumptions beyond weak convexity.

For weakly convex functions, apart from the classical stationarity measure $\dist(0,\partial f(x))$, another standard stationarity measure is $\|\nabla f_\gamma(x)\|$, which is related to the \emph{Moreau envelope} and the \emph{proximal map} \cite{Moreau1965,Rockafellar1976}:
\[
    f_\gamma(x):=\inf_y\Bigl\{f(y)+\tfrac{1}{2\gamma}\|y-x\|^2\Bigr\},\qquad
    \prox_{\gamma f}(x)\in\argmin_y\Bigl\{f(y)+\tfrac{1}{2\gamma}\|y-x\|^2\Bigr\}.
\]
When $\gamma\rho<1$, the proximal map is single-valued, the envelope $f_\gamma$ is $C^1$-smooth, and $\nabla f_\gamma(x)=\gamma^{-1}\bigl(x-\prox_{\gamma f}(x)\bigr)$. Driving $\|\nabla f_\gamma(x)\|$ to zero certifies that $x$ is close to a point that is approximately stationary for $f$. This envelope-based viewpoint underlies modern complexity results for weakly convex problems, both deterministic and stochastic \cite{Drusvyatskiy2018survey,DDStochModelBased,DavisGrimmerProximallyGuided,BoehmWrightVariableSmoothing}. 

Proximal-based methods operate by repeatedly solving the regularized subproblem $\min_y\{f(y)+\tfrac{1}{2\gamma}\|y-x_k\|^2\}$ and updating $x_{k+1}$ from its solution. Both their convergence theory and the Moreau-envelope stationarity rely on $\gamma$ being chosen on a \emph{safe} scale relative to $\rho$. For the analysis we use the conservative safe threshold
\[
    \bar\gamma := \frac{1}{2\rho},
\]
and call parameters
\(\gamma\le\bar\gamma\) safe. Our algorithm never uses this threshold; it is only an analytical boundary separating the regime where
standard proximal theory applies from the regime where the subproblem may be
nonconvex. We use the more conservative threshold $1/(2\rho)$ to obtain uniform constants in the analysis. In practice $\rho$ is rarely known: it depends on global properties of $f$, varies across data distributions, and is expensive to estimate. Overestimating $\rho$ yields conservatively small proximal parameters and slow progress, while  underestimating it can push the method into the unsafe regime, where the analysis breaks down.

The unknown safe threshold already creates a difficulty in the deterministic
setting: the algorithm cannot know whether a given proximal parameter \(\gamma_k\) places the
subproblem in the safe regime, where the subproblem is strongly convex. Figure~\ref{fig:two_regime_difficulty} illustrates this dichotomy. In the stochastic setting, this
difficulty is amplified. 
The proximal subproblem can only be solved approximately, and the function decrease must itself be estimated from noisy samples. Thus the algorithm must decide
whether to accept a candidate and how to update \(\gamma_k\) using only directly observable, noisy descent information, without access to the true subproblem gap or the safe/unsafe boundary. 

\begin{figure}[!htbp]
    \centering
    \definecolor{niceblue}{RGB}{40,90,200}
    \resizebox{\linewidth}{!}{%
    \begin{tikzpicture}[
        every node/.style={font=\footnotesize},
        safebox/.style={
            draw=niceblue, thick, rounded corners=4pt,
            fill=niceblue!8, text width=5.7cm, align=left,
            inner sep=8pt
        },
        unsafebox/.style={
            draw=niceblue, thick, rounded corners=4pt,
            fill=niceblue!8, text width=5.7cm, align=left,
            inner sep=8pt
        },
    ]
        \node[safebox, anchor=south] at (-3.45, 3.4) {%
            \textbf{Safe regime}\quad$\gamma \le \bar\gamma$\\[3pt]
            $\bullet$\;\,Subproblem is strongly convex\\
            $\bullet$\;\,Standard prox-point theory applies\\
            $\bullet$\;\,Solvable with guarantee
        };
        \node[unsafebox, anchor=south] at (3.45, 3.4) {%
            \textbf{Unsafe regime}\quad$\gamma > \bar\gamma$\\[3pt]
            $\bullet$\;\,Subproblem may be nonconvex\\
            $\bullet$\;\,Standard theory breaks down\\
            $\bullet$\;\,No accuracy guarantee
        };

        \draw[->, very thick] (-7, 2.9) -- (7, 2.9) node[right] {$\gamma$};
        \draw[very thick] (-6.8, 2.78) -- (-6.8, 3.02);
        \node[below=2pt] at (-6.8, 2.9) {$0$};

        \draw[very thick] (0, 2.75) -- (0, 3.05);
        \node[below=2pt, font=\footnotesize\bfseries] at (0, 2.78) {$\bar\gamma = 1/(2\rho)$};
        \node[below=2pt, font=\scriptsize\itshape, gray!55!black] at (0, 2.40) {(algorithm cannot see the safe vs unsafe boundary)};
    \end{tikzpicture}%
    }
    \caption{The two-regime difficulty. The proximal subproblem behaves qualitatively differently across the threshold $\bar\gamma=1/(2\rho)$, but the algorithm cannot observe this threshold when $\rho$ is unknown. }
    \label{fig:two_regime_difficulty}
\end{figure}

APS treats the safe threshold as an analysis device rather than as an input.
At iteration $k$, the method computes one candidate point $x_k^+$ for the
current proximal subproblem, estimates the function decrease
$f(x_k)-f(x_k^+)$, and accepts the candidate only if the observed decrease is
large enough relative to the proximal displacement. Acceptance increases
$\gamma_k$, allowing larger and less regularized steps; rejection decreases
$\gamma_k$, pushing the method back toward the safe regime.
The analysis separates the run into two regimes. On safe iterations,
$\gamma_k\le \bar\gamma$, the proximal displacement
$d_k=\gamma_k^{-1}(x_k-x_k^+)$ can be related to the reference Moreau-envelope
gradient. On unsafe iterations, no convexity or proximal-oracle accuracy is
assumed; these iterations are controlled only through observable descent and
the induced dynamics of $\gamma_k$. Intuitively, accepted unsafe steps are paid for by observed decrease,
while rejected unsafe iterations reduce \(\gamma_k\) and therefore drive the method
back toward the safe regime. Careful analysis linking these two regimes is the main mechanism behind both the deterministic and stochastic guarantees.

The bidirectional update in APS is different from a fixed monotone stepsize schedule and from a one-sided backtracking rule. Since the weak
convexity parameter is unknown, the algorithm does not know a priori whether the current proximal parameter is too aggressive or too conservative. Small values of $\gamma$ make the proximal subproblems more strongly regularized, but they may also produce overly conservative steps. Larger values of $\gamma$ can yield substantially more progress when the resulting candidate still satisfies
the observed descent test. APS therefore increases $\gamma$ after successful iterations, allowing it to exploit such productive regimes, and decreases $\gamma$ after rejected iterations, returning toward safer, more strongly regularized subproblems. Thus the proximal parameter is learned online from observable function-decrease information rather than fixed in advance or forced
to move in only one direction. \autoref{fig:aps_intro} previews the payoff on a robust phase-retrieval instance: 
APS reaches a lower recovery error than the fixed safe step (left) and removes the narrow tuning window of the proximal parameter altogether (right).

\begin{figure}[tbp]
  \centering
  \includegraphics[width=\linewidth]{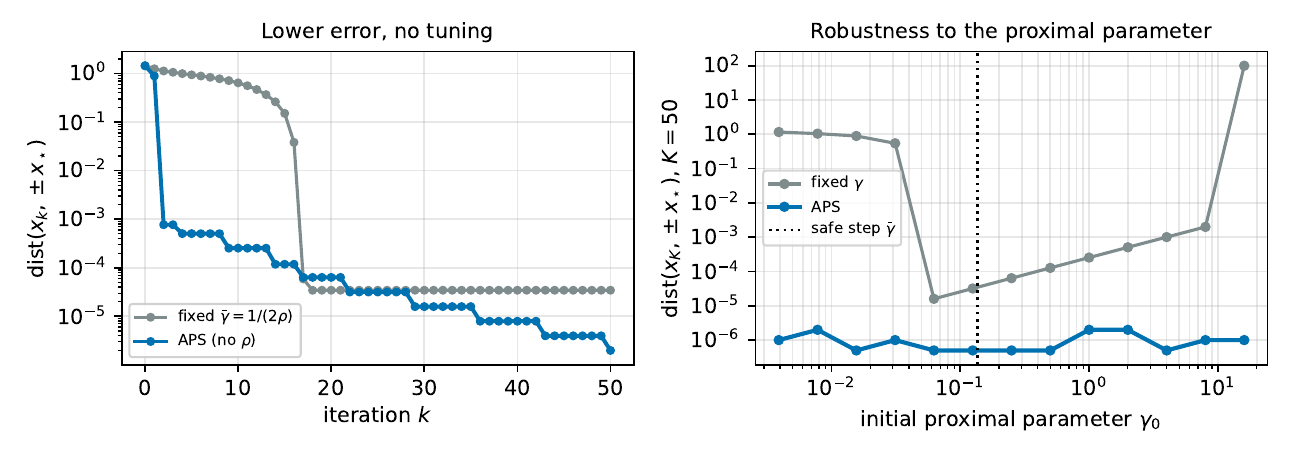}
  \caption{Preview of APS on robust phase retrieval, deterministic setting (full study in \autoref{sec:numerics}). \emph{Left:} APS ($\gamma_0=1$, no $\rho$) reaches a lower recovery error than the fixed safe step $\bar\gamma=1/(2\rho)$. \emph{Right:} robustness to the initial proximal parameter -- the good proximal parameter window is narrow, while APS beats the best fixed $\gamma$ without any tuning.}
  \label{fig:aps_intro}
\end{figure} 

\paragraph{Contributions.}  
We propose the Adaptive Prox-Guided Scheme (APS){ as a general proximal-based framework} for minimizing general deterministic or stochastic weakly convex functions. To the best of our knowledge, APS is the first framework that simultaneously handles an unknown weak-convexity parameter, avoids global Lipschitz or smoothness assumptions beyond weak convexity, and adapts the proximal parameter online and
bidirectionally.
In the stochastic setting, high-probability guarantees are obtained under deliberately weak oracle assumptions.
\begin{itemize}
    \item \emph{Adaptive proximal framework without known $\rho$.}
    APS uses one candidate-oracle call and one function-difference estimate per
    iteration, and updates the proximal parameter bidirectionally through the estimated progress of the algorithm. A single unified scheme serves both the deterministic and stochastic settings: they differ only in the quality and reliability of the oracles used. The safe threshold $\bar\gamma$ is never used by the algorithm, entering only the analysis. 
	\item {\emph{General Model-based framework.}
	APS covers adaptive proximal-point for general weakly convex functions, adaptive prox-linear methods for composite objectives
	$f=h\circ c$, and adaptive proximal-gradient methods for additive objectives $f=g+h$.
	More generally, the analysis also extends to any method that admits a model-based candidate oracle.}

	\item \emph{Two-regime analysis.}
    The analysis separates safe iterations, where the Moreau-gradient proxy can be
    controlled, from unsafe iterations, where no convexity or accuracy guarantee is imposed on the candidate. The analysis does not require the algorithm to identify the safe regime.

    \item \emph{Deterministic guarantee.}
    With exact function differences and an exact proximal oracle on safe calls \mbox{$\gamma\le\bar\gamma$} (and arbitrary output otherwise), APS
    obtains an $\mathcal O(\varepsilon^{-2})$ iteration complexity for
    subgradient stationarity. This matches the standard $\mathcal O(\varepsilon^{-2})$ iteration complexity required to reach $\varepsilon$-stationarity of a general weakly convex function.

    \item \emph{Stochastic guarantee under weak oracle assumptions.} We then analyze APS when the algorithm only has access to biased, heavy-tailed estimates of function differences and a candidate oracle that
    is reasonably accurate only with constant probability on safe calls, and can be arbitrary otherwise. 
	We prove that APS obtains a
    high-probability iteration complexity of $\mathcal O(\varepsilon^{-2})$ for the
    Moreau-envelope stationarity measure. For expected-risk objectives, the
    abstract oracles are realized using paired minibatch differences and
    stochastic-subgradient inner solves. 
\end{itemize}

In Table~\ref{tab:comparison} below, we summarize the main differences between the deterministic and stochastic settings, for the proximal-point version of APS. The more general model-based version of APS is discussed later in \autoref{sec:model_oracles}.
\begin{table}[htbp]
    \centering
    \caption{Comparison of the deterministic and stochastic proximal-point version of APS.}
    \label{tab:comparison}
    \renewcommand{\arraystretch}{1.5}
    \begin{tabular}{|L{2.6cm}|L{5.0cm}|L{5.0cm}|}
        \hline
        \textbf{Feature} & \textbf{Deterministic (Sec.~\ref{sec:det_specialization})} & \textbf{Stochastic (Sec.~\ref{sec:inexact_backtracking_proximal_point})} \\ \hline
        \textbf{Stationarity measure} & $\dist(0,\partial f(x))\le\varepsilon$ & $\|\nabla f_{\bar\gamma}(x)\|\le\varepsilon$ \\ \hline
        \textbf{Oracle model} & Exact proximal operator $\prox_{\gamma f}(x)$ on safe calls, otherwise the candidate oracle output can be arbitrary; exact function differences & Stochastic proximal oracle (SPO), reasonably accurate with constant probability on safe calls, {otherwise allowed to return an arbitrary point in $\dom f$}; stochastic difference oracle (SDO) with biased, heavy-tailed errors \\ \hline
        \textbf{Iteration complexity} & $\mathcal O(\varepsilon^{-2})$, deterministic & $\mathcal O(\varepsilon^{-2})$, high probability \\ \hline
    \end{tabular}
\end{table}

\subsection{Related work}\label{sec:related_work}

Our work contributes to the literature on proximal-based methods for weakly convex optimization. Weak convexity is a standard structural assumption in nonsmooth nonconvex optimization, with analyses often built on the Moreau envelope and proximal-point  methodology \cite{Moreau1965,Rockafellar1976}. Modern analyses commonly measure progress through the Moreau-envelope gradient at a safe smoothing parameter rather than through the subgradient distance $\dist(0,\partial f(x))$ \cite{Drusvyatskiy2018survey,DDStochModelBased,DavisGrimmerProximallyGuided}.

One important weakly convex class is the composite class $f(x)=h(c(x))$ with $h$ convex and Lipschitz and $c$ smooth. In \cite{DrusvyatskiyPaquette2018}, the authors analyzed the deterministic prox-linear method for the sum of a convex function and such a composition, and established an $O(\varepsilon^{-2})$ iteration complexity. For objectives consisting of a smooth finite-sum part plus a convex, possibly nonsmooth regularizer, the 4WD-Catalyst framework of \cite{PaquetteCatalyst} wraps gradient-based convex-optimization solvers in proximal subproblems and automatically adapts the proximal regularization upward to an unknown weak-convexity constant.
The prox-linear and proximal-point approaches were subsequently extended to the stochastic setting. In \cite{DuchiRuan2018}, the authors developed a family of stochastic model-based methods, including a stochastic prox-linear method, and established asymptotic almost-sure convergence. In \cite{DavisGrimmerProximallyGuided}, the authors proposed the proximally guided stochastic subgradient method, which applies gradient descent to the Moreau envelope with each proximal point evaluated inexactly by an inner stochastic subgradient solver; the method attains the standard $O(\varepsilon^{-4})$ sample complexity. Davis and Drusvyatskiy \cite{DDStochModelBased} developed a stochastic model-based framework, proving $O(\varepsilon^{-4})$ sample complexity in expectation for Moreau-envelope stationarity under a known weak-convexity parameter.

Subsequent work broadened the weakly convex optimization literature in several directions. In \cite{MaiJohansson2020}, the authors obtained $O(\varepsilon^{-4})$ expected sample-complexity guarantees for a stochastic heavy-ball method under a bounded stochastic-subgradient assumption. In \cite{AlacaogluMalitskyCevher2021}, the authors analyzed AMSGrad for constrained stochastic weakly convex problems and obtained an expected $O(\varepsilon^{-4})$ sample-complexity up to logarithmic factors. The work \cite{BoehmWrightVariableSmoothing} studies a deterministic variable smoothing scheme for the sum of a smooth function and the composition of a weakly convex function with a linear map. The weakly convex component is replaced by its Moreau envelope with a $\rho$-dependent schedule, and an $O(\varepsilon^{-3})$ iteration complexity is obtained. The work \cite{GaoDeng2024} extends stochastic model-based methods beyond global Lipschitz continuity, and \cite{DengGao2025} gives a smooth-approximation framework covering Nesterov and Moreau smoothing, with deterministic and stochastic guarantees based on target-driven smoothing levels. The  work \cite{HongLin2025} proves high-probability rates for AdaGrad-Norm and an AGD template on constrained Lipschitz weakly convex
problems under bounded or sub-Gaussian noise. The recent work \cite{LiaoZheng2025} proposes a deterministic proximal descent method with bundle updates that accelerates under smoothness or quadratic growth.

Most existing methods for weakly convex optimization require the weak-convexity parameter $\rho$ to be known to the algorithm, and removing this requirement has attracted growing interest. Several works discussed above give the first attempts on this problem. The backtracking variant of the algorithm in \cite{DrusvyatskiyPaquette2018} removes this requirement using line search, but requires deterministic evaluations and is tied to the composite structure. The stochastic prox-linear method of \cite{DuchiRuan2018} runs on a predetermined schedule, and offers an asymptotic guarantee. The parameter-free variant of \cite{DavisGrimmerProximallyGuided} removes knowledge of $\rho$ through a prescribed monotone schedule, though with a weaker worst-case 
guarantee. The work \cite{PaquetteCatalyst} adapts the proximal regularization upward to an unknown weak-convexity constant, doubling it whenever a convexity test fails. The AdaGrad-type methods \cite{AlacaogluMalitskyCevher2021,HongLin2025} build their decreasing stepsizes online from observed gradients, and the heavy-ball method of \cite{MaiJohansson2020} uses a prescribed schedule free of problem constants. These methods usually require additional structure beyond weak convexity: for example, Lipschitz-type assumptions \cite{DavisGrimmerProximallyGuided,AlacaogluMalitskyCevher2021,HongLin2025,MaiJohansson2020}, or smoothness \cite{PaquetteCatalyst}. In this paper, we adapt the proximal parameter itself online and bidirectionally, increasing and decreasing it through a verifiable test without committing to a schedule, and obtain high-probability finite-time guarantees for general weakly convex functions that need not be smooth or Lipschitz, under heavy-tailed value noise and a merely constant-probability proximal oracle.  As a result, we extend the scope of previous works by allowing significantly more relaxed assumptions on both the problem structure and the inputs of the algorithm.

Methodologically, our framework builds on a line of work on adaptive optimization with probabilistic and noisy oracles, in which trust-region, line-search, and cubic regularized Newton steps are accepted or rejected based on inexact information, with complexity guarantees in expectation or with high probability \cite{CartisScheinberg2018,ChenMenickellyScheinberg2018,GrattonRoyerVicenteZhang2018,BlanchetCartisMenickellyScheinberg2019,PaquetteScheinberg2020,CaoBerahasScheinberg2024}. Within this line, our prior work analyzes adaptive line-search and step-search methods \cite{JinScheinbergXieLineSearch2021,JinScheinbergXieStepSearch2024}, establishes a general sample-complexity framework for adaptive algorithms with stochastic oracles \cite{JinScheinbergXie2025}, develops first- and second-order adaptive cubic-regularization methods \cite{ScheinbergXieCubics2026}, and unifies these under a single high-probability analysis that tolerates unreliable oracle inputs \cite{ScheinbergXieUnified2025}; related developments address nonlinear-equality-constrained stochastic optimization \cite{BerahasXieZhou2025} and stochastic quasi-Newton methods \cite{MenickellyWildXie2026}. All of the previous work in this line requires the function to be smooth. This paper brings this idea to nonsmooth weakly convex problems, where we face a difficulty with no smooth analogue: the proximal parameter controls whether the subproblem is even tractable, and the safe/unsafe boundary, determined by the unknown $\rho$, is invisible to the algorithm.

\paragraph{Paper Organization.}
The rest of the paper is organized as follows. In Section~\ref{sec:prelims}, we introduce the problem setting, review basic properties of the Moreau envelope, and introduce the deterministic and stochastic settings in the paper. In Section~\ref{sec:unified}, we present the APS scheme, which is shared by the deterministic and stochastic settings. Sections~\ref{sec:det_specialization} and~\ref{sec:inexact_backtracking_proximal_point} then analyze the proximal-point version of APS, establishing $O(\varepsilon^{-2})$ iteration complexity in the deterministic setting and the same iteration complexity with high probability in the stochastic setting. Section~\ref{sec:model_oracles} then extends the same APS mechanism to general model-based candidates by abstracting the candidate step as the solution of a regularized local model subproblem. Adaptive proximal-point, prox-linear, and proximal-gradient methods arise as special cases of this model-based APS framework with different choices of the model. In Section~\ref{sec:sample_complex}, we apply the APS framework to expected-risk objectives and provide concrete constructions of the difference and candidate oracles. Section~\ref{sec:numerics} then illustrates the performance of APS numerically in the deterministic and stochastic settings.

	\section{Problem Setting}\label{sec:prelims}

	We consider the  minimization problem
	\[ \min_{x \in \mathbb{R}^d} f(x), \]
	where $f: \mathbb{R}^d \to \mathbb{R}{\cup\{+\infty\}}$ is proper and lower semicontinuous, {which allows constrained objectives ($f(x)=+\infty$ if $x$ is not in the feasible set)}. We make the following standing assumptions.

	\begin{assumption}[Weak Convexity]\label{ass:weak_convexity}
		The function $f$ is $\rho$-weakly convex for some $\rho > 0$. That is, the function $x \mapsto f(x) + \frac{\rho}{2}\|x\|^2$ is convex.
	\end{assumption}
	
	\begin{assumption}[Boundedness]\label{ass:boundedness}
		The function $f$ is bounded below by a finite value $f_{\inf} > -\infty$.
	\end{assumption}
	As stated in \autoref{sec:intro}, the weakly convex function class is broad. It includes convex functions, whose minimal weak-convexity constant is $0$ but which are $\rho$-weakly convex for every $\rho>0$, $L$-smooth functions, which are $L$-weakly convex, composite constructions $f(x)=h(c(x))$ with $h$ $L_h$-Lipschitz convex and $c$ smooth with $L_{\nabla c}$-Lipschitz Jacobian, which are $L_hL_{\nabla c}$-weakly convex \cite{DrusvyatskiyPaquette2018,DDStochModelBased}, and constrained objectives $g+\iota_C$ with $g$ weakly convex and $C$ closed convex. In this work, we do not assume that any valid weak-convexity constant is known to the algorithm, nor do we require $f$ to be Lipschitz or smooth. This makes the method applicable to a wide range of problems where the weak-convexity  parameter is unknown or hard to estimate, and where function values grow faster than linearly.

	\paragraph{Subdifferential.}
	Since $f$ is $\rho$-weakly convex, the function $h(x) := f(x) + \frac{\rho}{2}\|x\|^2$ is convex. We define the \emph{subdifferential} of $f$ at $x$ by
	\[ \partial f(x) := \{ v \in \mathbb{R}^d : v + \rho x \in \partial h(x) \}, \]
	where $\partial h(x)$ denotes the convex subdifferential of $h$. This definition coincides with the Fr\'echet subdifferential of $f$ (see, e.g., \cite{RW98,Drusvyatskiy2018survey}). Equivalently, {for $x\in\dom f$,} $v \in \partial f(x)$ if and only if
	\[ f(y) \ge f(x) + \langle v, y-x \rangle - \frac{\rho}{2}\|y-x\|^2, \quad \forall y \in \mathbb{R}^d. \]
	We refer to elements of $\partial f(x)$ as subgradients.

	Under these assumptions the proximal objective $f(\cdot)+\tfrac{1}{2\gamma}\|\cdot-x\|^2$ is proper, lower semicontinuous, and coercive, so it attains its minimum for every $\gamma>0$ by Weierstrass' theorem. We therefore write
	\[
	\begin{gathered}
	\prox_{\gamma f}(x) \in \argmin_{y} \left\{ f(y) + \frac{1}{2\gamma}\|y-x\|^2 \right\}, \\
	f_\gamma(x) := \min_{y} \left\{ f(y) + \frac{1}{2\gamma}\|y-x\|^2 \right\}.
	\end{gathered}
	\]
	When $\gamma < 1/\rho$, the minimizer is unique, the envelope $f_\gamma$ is $C^1$ smooth, and $\nabla f_\gamma(x) = \gamma^{-1}(x - \prox_{\gamma f}(x))$ {\cite[Lemma 2.2]{DDStochModelBased}}. For larger $\gamma$ the proximal subproblem is no longer strongly convex; the minimum is still attained but may be non-unique, and the Moreau envelope need not be differentiable.

	An important component of the framework is the \emph{candidate oracle}. We will mainly focus on the proximal oracle as the candidate oracle in \autoref{sec:prelims}-\ref{sec:inexact_backtracking_proximal_point}, so the safe-regime properties and Moreau-envelope tools below are stated for the proximal map. \autoref{sec:model_oracles} later generalizes the candidate oracle to model-based oracles that capture the prox-linear and proximal-gradient methods.

A central quantity throughout the analysis is the \emph{safe-regime threshold}
	\[
	\bar\gamma := \frac{1}{2\rho}.
	\]
It is the upper end of what we call the \emph{safe regime} $\gamma\le\bar\gamma$. We emphasize that $\bar\gamma$ is used \emph{only in the analysis}: the algorithm does not know $\rho$ or $\bar\gamma$. The choice $\bar\gamma = 1/(2\rho)$ leaves a constant-factor margin from the degeneracy boundary $\gamma\rho = 1$ at which the proximal subproblem loses strong convexity; whenever $\gamma\le\bar\gamma$, the proximal subproblem is $(1/\gamma-\rho)$-strongly convex with modulus at least $\rho$, the proximal point $\prox_{\gamma f}(x)$ is unique, and $f_\gamma$ is differentiable. We call such $\gamma$ a \emph{safe} proximal parameter, and any $\gamma>\bar\gamma$ an \emph{unsafe} proximal parameter.

	We consider two distinct settings for finding an approximate stationary point, depending on the available information:
 
	\paragraph{1. Deterministic Setting.}
	In this setting, we assume access to exact function values, and an oracle that can evaluate $\prox_{\gamma f}(x)$ whenever $\gamma\le\bar\gamma$. For $\gamma > \bar\gamma$, we do not impose any requirements on the proximal point oracle, and its output can be arbitrary in \(\dom f\).
	In the deterministic setting, the goal is to find a point $x$ satisfying:
	\[ \dist(0, \partial f(x)) \le \varepsilon. \]
 
	\paragraph{2. Stochastic Setting.}
	In this setting, we assume access to a noisy function value difference oracle with possibly biased and heavy-tailed errors. Inside the safe regime $\gamma\le\bar\gamma$, the proximal oracle is only required to be sufficiently accurate with constant probability, but otherwise can be an arbitrarily bad estimate of $\prox_{\gamma f}(x)$. For $\gamma > \bar\gamma$, the output of the proximal oracle can again be arbitrary in \(\dom f\). In the stochastic setting, the goal is to find a point $x$ satisfying
	\[ \|\nabla f_{\bar\gamma}(x)\| \le \varepsilon. \]
	This is a standard measure for stationarity in weakly convex optimization: $\|\nabla f_{\bar\gamma}(x)\|\le\varepsilon$ implies that $x$ is close to a point $y$ (specifically $y=\prox_{\bar\gamma f}(x)$) that is nearly stationary ($\dist(0,\partial f(y))\le\varepsilon$). This measure is robust to local nonsmoothness. When $f$ is nonsmooth, the subgradient set $\partial f(x)$ can be unstable (e.g., changing discontinuously). In stochastic regimes, we cannot reliably find a point with small subgradient norm directly; instead, we target the gradient of the Moreau envelope at the fixed safe parameter $\bar\gamma$.

	\subsection{Basic Properties in the Safe Regime}
	Before stating the algorithmic scheme, we introduce four properties of the proximal map inside the safe regime that are used by both the deterministic and stochastic analyses. 

	\begin{lemma}\label{lem:safe_subgrad_cert}
		Let $\gamma>0$, $x\in\mathbb R^d$, and $y=\prox_{\gamma f}(x)$, and set $d:=(x-y)/\gamma$. Then $d\in\partial f(y)$, and consequently $\dist(0,\partial f(y))\le\|d\|$.
	\end{lemma}
	\begin{proof}
		The Fermat rule applied to the proximal subproblem at its minimizer $y$ gives $0\in\partial f(y)+\gamma^{-1}(y-x)$, so $d=\gamma^{-1}(x-y)\in\partial f(y)$. The distance bound is immediate.
	\end{proof}

	\begin{lemma}[Sufficient Descent]\label{lem:accept}
		Suppose Assumption~\ref{ass:weak_convexity} holds, and let $\sigma\in(0,1)$. For any $\gamma\le\bar\gamma$, $x\in \dom f$, and $y=\prox_{\gamma f}(x)$ with $d=(x-y)/\gamma$,
		\[
		f(y) \le f(x) - \frac{2-\rho\gamma}{2}\gamma\|d\|^2 \le f(x)-\frac{\sigma}{2}\gamma\|d\|^2.
		\]
	\end{lemma}
	\begin{proof}
		Let $g(u) = f(u) + \frac{1}{2\gamma}\|u-x\|^2$. Since $f$ is $\rho$-weakly convex and $\gamma\le\bar\gamma$, $g$ is strongly convex with parameter $\mu = 1/\gamma-\rho\ge\rho>0$, and $y$ is the unique minimizer of $g$. Strong convexity gives
		\[ g(x) \ge g(y) + \frac{\mu}{2}\|x-y\|^2. \]
		Substituting the definition of $g$:
		\[ f(x) \ge f(y) + \frac{1}{2\gamma}\|y-x\|^2 + \frac{1/\gamma-\rho}{2}\|x-y\|^2. \]
		Rearranging terms and using $\|x-y\| = \gamma\|d\|$:
		\[ f(x) - f(y) \ge \left(\frac{1}{2\gamma} + \frac{1}{2\gamma} - \frac{\rho}{2}\right)\gamma^2\|d\|^2 = \left(\frac{1}{\gamma} - \frac{\rho}{2}\right)\gamma^2\|d\|^2 = \frac{2-\rho\gamma}{2}\gamma\|d\|^2. \]
		The condition $f(y) \le f(x) - \frac{\sigma}{2}\gamma\|d\|^2$ holds if $\frac{2-\rho\gamma}{2} \ge \frac{\sigma}{2}$. Since $\gamma \leq 1/(2\rho)$, we have $\rho\gamma \leq 1/2$, so $2-\rho\gamma > 1 > \sigma$. Thus the inequality is always satisfied.
	\end{proof}

	\begin{lemma}\label{lem:consistency}
		Suppose Assumption~\ref{ass:weak_convexity} holds. Let $0<\gamma\le \bar\gamma$ and let
		\[
		\Phi(y):=f(y)+\frac{1}{2\gamma}\|y-x\|^2,
		\qquad
		\Phi^* := \min_{y}\Phi(y).
		\]
		If a candidate $y$ satisfies $\Phi(y)-\Phi^*\le\nu$, then
		\[
		\left\|\frac{x-y}{\gamma}-\nabla f_\gamma(x)\right\|
		\le
		\sqrt{\frac{2\nu}{\gamma(1-\rho\gamma)}}\le 2\sqrt{\frac{\nu}{\gamma}}.
		\]
		In particular, $\nu\le \varepsilon^2\gamma/16$ caps the perturbation by $\varepsilon/2$, so whenever $\|\nabla f_\gamma(x)\|>\varepsilon$, we have
		\[
		\left\|\frac{x-y}{\gamma}\right\|>\frac{\varepsilon}{2}.
		\]
	\end{lemma}
	\begin{proof}
		Set $\bar y=\prox_{\gamma f}(x)$. Since $f$ is $\rho$-weakly convex, $\Phi$ is $(1/\gamma-\rho)$-strongly convex, and $\bar y$ is its unique minimizer, so $\Phi^*=\Phi(\bar y)$. Strong convexity around $\bar y$ then gives
		\[
		\|y-\bar y\|^2
		\le
		\frac{2(\Phi(y)-\Phi^*)}{1/\gamma-\rho}\le\frac{2\nu}{1/\gamma-\rho}
		=
		\frac{2\nu\gamma}{1-\rho\gamma}.
		\]
		Since
		\[
		\frac{x-y}{\gamma}-\nabla f_\gamma(x)
		=
		\frac{\bar y-y}{\gamma},
		\]
		we have
		\[
		\left\|\frac{x-y}{\gamma}-\nabla f_\gamma(x)\right\|^2
		\le
		\frac{2\nu}{\gamma(1-\rho\gamma)}.
		\]
		Taking square roots gives the sharper perturbation bound, and the last inequality in the statement follows from \(\gamma\le\bar\gamma\). The second claim follows from the triangle inequality.
	\end{proof}

	\begin{lemma}\label{cor:moreau_factor2}
		Suppose Assumption~\ref{ass:weak_convexity} holds. If $0<\gamma\le\lambda\le \bar\gamma$, then for every $x\in\mathbb R^d$,
		\begin{equation}\label{eq:moreau_comp}
			\|\nabla f_\lambda(x)\|
			\le
			\frac{1-\gamma\rho}{1-\lambda\rho}\|\nabla f_\gamma(x)\|
			\le 2\|\nabla f_\gamma(x)\|.
		\end{equation}
	\end{lemma}

	\begin{proof}
		The case $\gamma=\lambda$ is immediate; assume $\gamma<\lambda$. Set $y_\lambda=\prox_{\lambda f}(x)$, $y_\gamma=\prox_{\gamma f}(x)$, $g_\lambda=\nabla f_\lambda(x)$, and $g_\gamma=\nabla f_\gamma(x)$. By optimality of the proximal points, $g_\lambda\in\partial f(y_\lambda)$ and $g_\gamma\in\partial f(y_\gamma)$, so hypomonotonicity of the subdifferential of a $\rho$-weakly convex function gives
		\[
		\langle g_\lambda-g_\gamma,\;y_\lambda-y_\gamma\rangle \ge -\rho\|y_\lambda-y_\gamma\|^2.
		\]
		Let $v:=y_\lambda-y_\gamma=\gamma g_\gamma-\lambda g_\lambda$. Substituting $g_\lambda=(\gamma g_\gamma-v)/\lambda$ into the inequality, multiplying by $\lambda>0$, and dividing by $\gamma-\lambda<0$ (which flips the inequality) yields
		\[
		\langle g_\gamma,v\rangle \le -\tfrac{1-\lambda\rho}{\lambda-\gamma}\|v\|^2.
		\]
		Together with $\langle g_\gamma,v\rangle\ge-\|g_\gamma\|\,\|v\|$, this gives $\|v\|\le\frac{\lambda-\gamma}{1-\lambda\rho}\|g_\gamma\|$ when $v\ne 0$. Therefore
		\[
		\lambda\|g_\lambda\| = \|\gamma g_\gamma-v\| \le \gamma\|g_\gamma\|+\|v\| \le \frac{\lambda(1-\gamma\rho)}{1-\lambda\rho}\|g_\gamma\|,
		\]
		so $\|g_\lambda\| \le \frac{1-\gamma\rho}{1-\lambda\rho}\|g_\gamma\| \le 2\|g_\gamma\|$, where the last inequality uses $1-\gamma\rho\le 1$ and $1-\lambda\rho\ge 1/2$. The case $v=0$ gives $\|g_\lambda\|=(\gamma/\lambda)\|g_\gamma\|$, which is also bounded by the factor because \(\gamma/\lambda\le(1-\gamma\rho)/(1-\lambda\rho)\) is equivalent to \(\gamma\le\lambda\).
	\end{proof}

	\section{The Adaptive Prox-Guided Scheme}\label{sec:unified}

	The two settings introduced in Section~\ref{sec:prelims} share a common algorithmic template, which we call the \emph{Adaptive Prox-Guided Scheme} (APS), presented in \autoref{alg:aps}. Each iteration of APS makes two oracle calls: one candidate-step query and one function-difference query. It then performs a single accept/reject check and updates the proximal parameter according to the outcome. The deterministic (\autoref{sec:det_specialization}) and stochastic (\autoref{sec:inexact_backtracking_proximal_point}) settings differ mainly in the oracles they use.
	
	APS adaptively updates the proximal parameter $\gamma>0$ during runtime. Throughout the paper, $\gamma_k$ denotes its value at iteration $k$. Below we introduce the key components of APS.
	
	\paragraph{Candidate and difference oracles.}
	Given the current iterate $x_k\in\dom f$ and parameter $\gamma_k>0$, the \emph{candidate oracle} returns a trial point $y_k \in \dom f$. In the proximal-point instance, $y_k$ approximates a minimizer of the proximal subproblem
	\[
		\min_y\left\{ f(y)+\frac{1}{2\gamma_k}\|y-x_k\|^2\right\}.
	\]
	In the general model-based instances (introduced in \autoref{sec:model_oracles}), $f$ is replaced by a local model $m_{x_k}$, and $y_k$ approximately solves the analogous subproblem with $m_{x_k}$ in place of $f$. The framework does not require an exact proximal step.
	The \emph{difference oracle} returns an estimate
	\[
		\Delta_k  \approx f(x_k)-f(y_k)
	\]
	of the decrease obtained by updating $x_k$ to $y_k$.
	
	In the deterministic setting, $\Delta_k\equiv f(x_k)-f(y_k)$ and the candidate oracle returns the exact proximal point whenever $\gamma \leq \bar{\gamma}$. If $\gamma_k > \bar{\gamma}$, its output can be arbitrary. In the stochastic setting, the difference oracle may be biased and have heavy-tailed errors. For the candidate oracle, in the safe regime $\gamma\le\bar\gamma$, it is only required to be sufficiently accurate with constant probability, but otherwise can be an arbitrarily bad estimate of $\prox_{\gamma f}(x)$. When $\gamma_k > \bar{\gamma}$, the output of the candidate oracle can again be an arbitrary point in \(\dom f\). The exact requirements for the stochastic oracles are stated in \autoref{sec:inexact_backtracking_proximal_point}.

	\paragraph{Moreau-gradient proxy.}
	We define the \emph{Moreau-gradient proxy} to be
	\[
		d_k := \frac{x_k-y_k}{\gamma_k}.
	\]
	If $\gamma_k\le\bar\gamma$ and $y_k=\prox_{\gamma_k f}(x_k)$, then $d_k$ is exactly the gradient of the Moreau envelope $f_{\gamma_k}$ at $x_k$, and Lemma~\ref{lem:safe_subgrad_cert} additionally gives $d_k\in\partial f(y_k)$. However, if $\gamma_k > \bar{\gamma}$, then $y_k$ can be arbitrary in \(\dom f\); and in the stochastic setting, $y_k$ is only guaranteed to be sufficiently close to $\prox_{\gamma_k f}(x_k)$ with constant probability, even when $\gamma_k \leq \bar{\gamma}$. For this reason we call $d_k$ a Moreau-gradient \emph{proxy} rather than a Moreau-envelope gradient. This Moreau-envelope view is specific to the proximal candidate; for the model-based candidates of \autoref{sec:model_oracles}, $d_k$ is not itself a Moreau-envelope proxy but we will see it still controls $\|\nabla f_{\gamma_k}(x_k)\|$.
	
	\paragraph{Accept/reject test.}
	A candidate point $y_k$ is accepted if it achieves sufficient estimated descent, up to oracle slack, and if $\norm{d_k}$ is not too small:
	\begin{equation}\label{eq:unified_test}
		\Delta_k \ge \frac{\sigma}{2}\gamma_k\|d_k\|^2 - \epsilon_f
		\qquad\text{and}\qquad
		\|d_k\| \ge \varepsilon_{\mathrm{rej}}.
		\tag{AT}
	\end{equation}
	Here $\sigma\in(0,1)$ is a fixed constant, $\varepsilon_{\mathrm{rej}}>0$ is a safeguard against accepting candidates with spuriously small residuals (without it, a bad SPO call could fool the test by returning $y_k = x_k$). We will later see that taking $\varepsilon_{\mathrm{rej}}=\mathcal O(\varepsilon)$ is a good choice. $\epsilon_f\ge 0$ is a slack term that accounts for error in the difference oracle. The deterministic setting sets $\epsilon_f=0$, while the stochastic setting may set $\epsilon_f$ to be an upper bound on the expected error of the stochastic difference oracle. If iteration $k$ satisfies \eqref{eq:unified_test}, we call it a \emph{successful} iteration; otherwise we call it an \emph{unsuccessful} iteration. We write 
	\[
		\Theta_k := \mathbf 1\{\text{iteration $k$ is successful}\}.
	\]
	
	\paragraph{Update rule.}
	Since $\rho$ is unknown, APS cannot fix a safe value of $\gamma$ a priori. Instead, it updates $\gamma$ bidirectionally according to the progress of the algorithm. A successful iteration increases $\gamma$ by a factor $\beta_{\mathrm{inc}}>1$, thereby allowing the scheme to take potentially larger steps in future iterates. An unsuccessful iteration decreases $\gamma$ by a factor $\beta_{\mathrm{dec}}\in(0,1)$, thereby shrinking the proximal parameter that may potentially be overly aggressive. Specifically,
	\begin{equation}\label{eq:unified_update}
		(x_{k+1},\gamma_{k+1}) =
		\begin{cases}
			(y_k,\beta_{\mathrm{inc}}\gamma_k), & \text{on a successful iteration},\\
			(x_k,\beta_{\mathrm{dec}}\gamma_k), & \text{on an unsuccessful iteration}.
		\end{cases}
	\end{equation}
	This update rule increases and decreases $\gamma_k$ adaptively based on the progress of the algorithm.  Intuitively, this allows $\gamma_k$ to adapt to the local geometry of the function. 
	
	\begin{algorithm}[htbp]
		\caption{Adaptive Prox-Guided Scheme (APS)}
		\label{alg:aps}
		\begin{algorithmic}[1]
			\State \textbf{Input:} initial iterate $x_0\in \dom f$; initial proximal parameter $\gamma_0>0$; descent parameter $\sigma\in(0,1)$; update factors $\beta_{\mathrm{inc}}>1$ and $\beta_{\mathrm{dec}}\in(0,1)$; rejection safeguard $\varepsilon_{\mathrm{rej}}>0$; oracle slack $\epsilon_f\ge 0$.
			\For{$k=0,1,2,\ldots$}
				\State Query the candidate oracle at $(x_k,\gamma_k)$ to obtain $y_k$.
				\State Set $d_k \gets (x_k-y_k)/\gamma_k$.
				\State Query the difference oracle at $(x_k,y_k)$ to obtain $\Delta_k\approx f(x_k)-f(y_k)$.
				\If{$\Delta_k\ge \frac{\sigma}{2}\gamma_k\|d_k\|^2-\epsilon_f$ \textbf{and} $\|d_k\|\ge \varepsilon_{\mathrm{rej}}$}
					\State $x_{k+1}\gets y_k$, $\gamma_{k+1}\gets \beta_{\mathrm{inc}}\gamma_k$ (\emph{successful} iteration).
				\Else
					\State $x_{k+1}\gets x_k$, $\gamma_{k+1}\gets \beta_{\mathrm{dec}}\gamma_k$ (\emph{unsuccessful} iteration).
				\EndIf
			\EndFor
		\end{algorithmic}
	\end{algorithm}

	We use the following notation in the remainder of the paper. We define the indicator
	\[
		U_k := \mathbf 1\{\max\{\gamma_k,\gamma_{k+1}\} > \bar\gamma\}.
	\]
	We call iteration $k$ an \emph{unsafe} iteration if $U_k = 1$ and a \emph{safe} iteration otherwise. A safe iteration means that both $\gamma_k$ and $\gamma_{k+1}$ lie in the safe regime. We define
	\[
		Z_k := f(x_k)-f_{\inf}
	\]
	as the measure of progress of the scheme.
	
	Set the log-scale update constant
	\[
		m := \frac{\ln\beta_{\mathrm{inc}}}{\ln(1/\beta_{\mathrm{dec}})}.
	\]
	The constant $m$ measures the size of one increase step in units of decrease steps. For clarity, we will assume for the rest of the paper that $m$ is a positive integer. The general $m \geq 1$ case can be handled by replacing $m$ with the corresponding floor or ceiling, but it makes the presentation messier without adding much insight.
	In the deterministic setting, there is no benefit to taking larger increases, and we use the reciprocal update choice $\beta_{\mathrm{inc}}=1/\beta_{\mathrm{dec}}$, so $m=1$. Larger values of $m$ are mainly useful in the stochastic setting, where reliable iterations may occur with small probability. A larger increase allows one successful reliable iteration to compensate for several decreases, preventing
	the proximal parameter from drifting toward zero, and in fact allows for an upward drift while in the safe regime.

	\begin{remark}[Practical parameter choices]\label{rem:practical_defaults}
		\sloppy
		All algorithmic parameters $(\sigma, \beta_{\mathrm{inc}}, \beta_{\mathrm{dec}}, \varepsilon_{\mathrm{rej}}, \gamma_0)$ are user-chosen and the analysis below holds for any valid choice. The bidirectional adaptation of $\gamma_k$ makes APS substantially more forgiving than any fixed-$\gamma$ scheme: $\gamma_0$ can be any positive scalar, and the analysis only pays a $\log$-factor for the gap between $\gamma_0$ and the safe scale $\bar\gamma$. Reasonable defaults are $\sigma=1/2$ (any $\sigma\in(0,1)$ works) and $\beta_{\mathrm{inc}}=1/\beta_{\mathrm{dec}}=2$. If the reliable-iteration probability is expected to be small, one may instead choose $m>1$, equivalently $\beta_{\mathrm{inc}}=(1/\beta_{\mathrm{dec}})^m$, so that each successful reliable iteration compensates for more unsuccessful decreases. 
	 It is natural to choose the safeguard threshold $\varepsilon_{\mathrm{rej}}$ proportional to the desired stationarity accuracy. For the deterministic proximal setting, we may take $\varepsilon_{\mathrm{rej}}=\varepsilon$, whereas in the 
	 stochastic proximal setting, we may take $\varepsilon_{\mathrm{rej}}=\frac{\varepsilon}{4}$, where $\varepsilon$ is the target accuracy for the optimization problem. The $\varepsilon_{\mathrm{rej}}$ used in the model-based extension is detailed in \autoref{sec:model_oracles}.
	\end{remark}

	Finally, define the logarithmic distances in units of decrease steps from the current parameter $\gamma_k$ to the safe-regime threshold $\bar\gamma$:
	\begin{equation}\label{eq:d0_def}
		\ell_k^+ := \max\left\{\left\lceil \frac{\ln(\gamma_k/\bar\gamma)}{\ln(1/\beta_{\mathrm{dec}})}\right\rceil,0\right\}, \quad
		\ell_k^- := \max\left\{\left\lfloor \frac{\ln(\bar\gamma/\gamma_k)}{\ln(1/\beta_{\mathrm{dec}})}\right\rfloor,0\right\}, \quad
		\ell_k := \ell_k^+ + \ell_k^-.
	\end{equation} 
	At most one of $\ell_k^+$ or $\ell_k^-$ can be nonzero: If $\gamma_k\le\bar\gamma$, then $\ell_k^+=0$, and if $\gamma_k\ge\bar\gamma$, then $\ell_k^-=0$. We think of $\ell_k$ as a measure of how far the proximal parameter $\gamma_k$ is from $\bar{\gamma}$.

	\section{{Deterministic Proximal Setting}}\label{sec:det_specialization}

	This section develops the analysis of APS in the deterministic setting where the candidate oracle returns the exact proximal point in the safe regime and an arbitrary candidate in the unsafe regime, and function differences $f(x)-f(y)$ are available exactly. 
	It is meaningful in proximal-friendly settings where the prox operator is computable. We will show that in this setting, APS finds a point $x$ with small $\dist(0,\partial f(x))$. 

	\subsection{The Oracles and Stationarity Measure}

	\paragraph{Safe-exact candidate oracle.}
	Given a query $(x,\gamma)$ with $\gamma>0$, the candidate oracle returns a point $y$. If $\gamma\le\bar\gamma$, then
	\[
	y = \prox_{\gamma f}(x).
	\]
	If $\gamma>\bar\gamma$, the oracle output can be an arbitrary point in $\dom f$, and hence no accuracy or descent condition is imposed on the output $y$.  

	The asymmetry between the two regimes is essential. Inside the safe regime the proximal subproblem is strongly convex and its unique minimizer can be obtained by any standard solver, and in some cases a closed form solution may exist. Outside the safe regime the same subproblem can be genuinely nonconvex, so a solver may return a candidate that need not be a global minimizer; we therefore impose no requirement on this candidate and let the descent test in APS decide whether the candidate is productive. 

	\paragraph{Exact difference oracle.} Given inputs $x$ and $y$, the difference oracle returns $\Delta=f(x)-f(y)$. 
	Since the difference oracle is exact, in the deterministic setting, we set $\epsilon_f=0$ in APS.

	\paragraph{Stationarity measure.} In the deterministic setting, the goal is to find a point $x$ with $\dist(0,\partial f(x))\le\varepsilon$. We define the stopping time
	\begin{equation}\label{eq:Tepsd}
		T_\varepsilon := \inf\bigl\{k\ge 0 : \dist(0,\partial f(y_k))\le\varepsilon\bigr\},
	\end{equation}
	the first iteration at which APS computes a point that is $\varepsilon$-subgradient-stationary.
	
	In this deterministic setting, as discussed in the last section, we set $\beta_{\mathrm{inc}}=1/\beta_{\mathrm{dec}}$, so $m=1$. The user picks a target accuracy $\varepsilon>0$ and runs \autoref{alg:aps} with the rejection safeguard set to the target accuracy: $\varepsilon_{\mathrm{rej}} = \varepsilon$ and set $\epsilon_f=0$.
	In the next section, we bound the stopping time $T_\varepsilon$ for APS.

	\subsection{Iteration Complexity}

	The iteration complexity analysis builds on two facts: every iteration in the safe regime before the stopping time is successful, and successful iterations with unsafe proximal parameters produce at least $h_{\mathrm{des}}$ guaranteed descent, where
	\begin{equation}\label{eq:h_des_def}
		h_{\mathrm{des}} := \frac{\sigma\bar\gamma\varepsilon_{\mathrm{rej}}^2}{2\beta_{\mathrm{inc}}}
		=\frac{\sigma\bar\gamma\varepsilon^2}{2\beta_{\mathrm{inc}}} = \frac{\sigma\varepsilon^2}{4\rho\beta_{\mathrm{inc}}},
	\end{equation}
	where the second equality uses the choice $\varepsilon_{\mathrm{rej}}=\varepsilon$ in the deterministic setting.

	\begin{theorem}[Deterministic setting: iteration complexity]\label{thm:exact_complexity}
		Suppose Assumptions~\ref{ass:weak_convexity} and~\ref{ass:boundedness} hold. For \autoref{alg:aps} with $\varepsilon_{\mathrm{rej}}=\varepsilon$  and $\beta_{\mathrm{inc}}=1/\beta_{\mathrm{dec}}$, we have
		\begin{equation}\label{eq:Tepsd_bound}
			T_\varepsilon
			\;\le\;
			\ell_0 + \frac{2(f(x_0)-f_{\inf})}{h_{\mathrm{des}}}
			\;=\;
			 \ell_0 + \frac{8\rho\beta_{\mathrm{inc}}(f(x_0)-f_{\inf})}{\sigma\varepsilon^2}
			\;=\; \mathcal O(\varepsilon^{-2}).
		\end{equation}
	\end{theorem} 

	\begin{proof}
		Define 
		\[
		\tau_\varepsilon := \inf\{k\ge 0 : \gamma_k\le\bar\gamma \text{ and } \|d_k\|\le\varepsilon\}.
		\]
		For any index $k$ with $\gamma_k\le\bar\gamma$ and $\|d_k\|\le\varepsilon$, the safe-exact oracle gives $y_k=\prox_{\gamma_k f}(x_k)$, and Lemma~\ref{lem:safe_subgrad_cert} yields $d_k\in\partial f(y_k)$, so $\dist(0,\partial f(y_k))\le\|d_k\|\le\varepsilon$. Hence $T_\varepsilon\le\tau_\varepsilon$, and to establish~\eqref{eq:Tepsd_bound} it suffices to show
		\[
		\tau_\varepsilon \le \ell_0 + 2(f(x_0)-f_{\inf})/h_{\mathrm{des}}.
		\]
		Fix any integer $t\le \tau_\varepsilon$. By definition of $\tau_\varepsilon$, for every $k<t$, either $\gamma_k>\bar\gamma$ or $\|d_k\|>\varepsilon$.

		First, we show that safe iterations before the stopping time are successful.
		Consider any iteration $k<t$ with $\gamma_k\le\bar\gamma$. Since $k < \tau_\varepsilon$, we have $\|d_k\|>\varepsilon=\varepsilon_{\mathrm{rej}}$, so the rejection safeguard is met. Moreover, the safe-exact oracle gives $y_k=\prox_{\gamma_k f}(x_k)$, and Lemma~\ref{lem:accept} then gives
		\[
		f(x_k)-f(y_k) \ge \frac{\sigma}{2}\gamma_k\|d_k\|^2.
		\]
		Since $\Delta_k = f(x_k)-f(y_k)$ and $\epsilon_f=0$, the sufficient descent condition holds. Therefore, iteration $k$ is successful. Hence, for every $k<t$,
		\begin{equation}\label{eq:safe_implies_succ_det}
			(1-U_k)(1-\Theta_k) = 0.
		\end{equation}
		Next, we lower bound the number of unsafe successful iterations using a counting argument.
		Set
		\begin{itemize}
			\item $N^{\rm unsafe}_+ := \sum_{k=0}^{t-1}U_k\Theta_k$ (number of unsafe successful iterations),
			\item $N^{\rm unsafe}_- := \sum_{k=0}^{t-1}U_k(1-\Theta_k)$ (number of unsafe unsuccessful iterations),
			\item $N^{\rm safe}_+ := \sum_{k=0}^{t-1}\Theta_k(1-U_k)$ (number of safe successful iterations).
		\end{itemize}
		By~\eqref{eq:safe_implies_succ_det} safe unsuccessful iterations do not occur, so 
		\[
		t = N^{\rm safe}_+ + N^{\rm unsafe}_+ + N^{\rm unsafe}_-.
		\] 
		Define $r_k := \ln(\gamma_k/\bar\gamma)/\ln(1/\beta_{\mathrm{dec}})$, so that the distance $\ell_k$ from~\eqref{eq:d0_def} is
		\[
		\ell_k = \max\{\lceil r_k\rceil,\,0\} + \max\{\lfloor -r_k\rfloor,\,0\}.
		\]
		This equals $\lceil r_k\rceil$ when $\gamma_k>\bar\gamma$ and $\lfloor -r_k\rfloor$ when $\gamma_k\le\bar\gamma$, so $\ell_k\ge 0$. 

		A successful iteration sends $r_{k+1}=r_k+1$ and an unsuccessful one sends $r_{k+1}=r_k-1$. We check how $\ell_k$ changes in all of the possible cases:
		\begin{itemize}
			\item \emph{Unsafe successful}: either $\gamma_k\le\bar\gamma<\gamma_{k+1}$ with $r_k\in(-1,0]$, giving $\ell_k=0$ and $\ell_{k+1}=\lceil r_k+1\rceil=1$; or $\gamma_k>\bar\gamma$, giving $\ell_k=\lceil r_k\rceil$ and $\ell_{k+1}=\lceil r_k+1\rceil=\ell_k+1$. In either case $\ell_{k+1}=\ell_k+1$.
			\item \emph{Unsafe unsuccessful} ($\gamma_k>\bar\gamma$): either $\gamma_{k+1}\le\bar\gamma$ with $r_k\in(0,1]$, giving $\ell_k=1$ and $\ell_{k+1}=\lfloor 1-r_k\rfloor=0$; or $\gamma_{k+1}>\bar\gamma$ with $r_k>1$, giving $\ell_k=\lceil r_k\rceil$ and $\ell_{k+1}=\lceil r_k-1\rceil=\ell_k-1$. In either case $\ell_{k+1}=\ell_k-1$.
			\item \emph{Safe successful} ($\gamma_k,\gamma_{k+1}\le\bar\gamma$): $r_k\le -1$ and $r_{k+1}=r_k+1\le 0$, so $\ell_{k+1}=\lfloor -r_k-1\rfloor=\ell_k-1$.
		\end{itemize}
		Summing,
		\[
		\ell_t-\ell_0 = N^{\rm unsafe}_+ - N^{\rm safe}_+ - N^{\rm unsafe}_-,
		\]
		and since $\ell_t\ge 0$,
		\begin{equation}\label{eq:t_bound_det}
			 t = N^{\rm safe}_+ + N^{\rm unsafe}_+ + N^{\rm unsafe}_- \le 2N^{\rm unsafe}_+ + \ell_0.
		\end{equation}

		Finally, we show that unsafe successful iterations produce sufficient descent.
		If $U_k\Theta_k=1$, then $\gamma_k\ge\bar\gamma/\beta_{\mathrm{inc}}$, the rejection safeguard $\|d_k\|\ge\varepsilon_{\mathrm{rej}}=\varepsilon$ holds, and the sufficient descent test gives $f(x_k)-f(y_k)  \ge \frac{\sigma}{2}\gamma_k\|d_k\|^2$. Therefore
		\[
		f(x_k)-f(x_{k+1}) \ge \frac{\sigma}{2}\gamma_k\varepsilon^2 \ge \frac{\sigma\bar\gamma\varepsilon^2}{2\beta_{\mathrm{inc}}} = h_{\mathrm{des}}.
		\]
		All other iterations contribute nonnegative descent. Summing the above inequality and using $f(x_t)\ge f_{\inf}$,
		\begin{equation}\label{eq:descent_det}
			h_{\mathrm{des}}\,N^{\rm unsafe}_+ \le f(x_0)-f_{\inf}.
		\end{equation}
		Substituting~\eqref{eq:descent_det} into~\eqref{eq:t_bound_det} gives $t \le \ell_0 + 2(f(x_0)-f_{\inf})/h_{\mathrm{des}}$ for every integer $t\le \tau_\varepsilon$. Thus
		\[
		\tau_\varepsilon\le \ell_0 + \frac{2(f(x_0)-f_{\inf})}{h_{\mathrm{des}}},
		\]
		which establishes~\eqref{eq:Tepsd_bound}.

	\end{proof}

	Theorem~\ref{thm:exact_complexity} bounds the stopping time $T_\varepsilon$. The following corollary provides a concrete way to obtain a point satisfying $\dist(0, \partial f(x)) \leq \varepsilon$ after running the algorithm. 

	\begin{corollary}\label{thm:det_offline}
		Suppose Assumptions~\ref{ass:weak_convexity} and~\ref{ass:boundedness} hold, and run \autoref{alg:aps} with $\varepsilon_{\mathrm{rej}}=\varepsilon$, and $\beta_{\mathrm{inc}}=1/\beta_{\mathrm{dec}}$. For any horizon $N$ satisfying
		\begin{equation}\label{eq:N_horizon}
			N > \ell_0 + \frac{2(f(x_0)-f_{\inf})}{h_{\mathrm{des}}},
		\end{equation}
		the set $\mathcal K_N := \{k<N : \|d_k\|\le\varepsilon\}$ is nonempty, and the output
		\[
		x_{\rm out} := y_{k_*},\qquad k_* \in \arg\min_{k\in\mathcal K_N}\gamma_k,
		\]
		satisfies $\dist(0,\partial f(x_{\rm out}))\le\varepsilon$.
	\end{corollary}

	\begin{proof}
		The proof of Theorem~\ref{thm:exact_complexity} bounds the safe small-residual hitting time $\tau_\varepsilon := \inf\{k\ge 0 : \gamma_k\le\bar\gamma,\ \|d_k\|\le\varepsilon\}$ by the right-hand side of~\eqref{eq:N_horizon}, so $\tau_\varepsilon < N$ and hence $\tau_\varepsilon\in\mathcal K_N$. Since $k_*$ minimizes $\gamma_k$ over $\mathcal K_N$, we have $\gamma_{k_*}\le\gamma_{\tau_\varepsilon}\le\bar\gamma$. The safe-exact oracle then gives $y_{k_*}=\prox_{\gamma_{k_*}f}(x_{k_*})$, and Lemma~\ref{lem:safe_subgrad_cert} gives $d_{k_*}\in\partial f(y_{k_*})$. Therefore $\dist(0,\partial f(x_{\rm out}))\le\|d_{k_*}\|\le\varepsilon$.
	\end{proof}

	 In words: after running \autoref{alg:aps}, we collect iterates with small residuals ($\|d_k\|\le\varepsilon$), pick the one with the smallest proximal parameter $\gamma_k$ among them, and return its candidate $y_k$ as $x_{\rm out}$. By Corollary~\ref{thm:det_offline}, this $x_{\rm out}$ satisfies $\dist(0,\partial f(x_{\rm out}))\le\varepsilon$. In practice, one can run the algorithm until the residual $\|d_k\|$ has been small for a while and then provide the output as discussed.
	
	\section{{Stochastic Proximal Setting}}\label{sec:inexact_backtracking_proximal_point}

	We now turn to the stochastic setting, where both the candidate step and the descent test must be carried out under noise. In this setting, \autoref{alg:aps} is run with a \emph{stochastic proximal oracle} (SPO) that is only required to be sufficiently accurate with a constant probability inside the safe regime $\gamma\le\bar\gamma$, and a \emph{stochastic difference oracle} (SDO) with possibly biased and heavy-tailed errors. Each iteration's accept/reject decision is now driven by noisy estimates rather than exact ones, and the analysis must accommodate both false acceptances and false rejections caused by oracle noise. In \autoref{sec:model_oracles}, we will reuse the arguments with a stochastic model-based oracle (SMO) in place of the SPO.
	
	\subsection{Stochastic Oracles}\label{sec:oracles}
	 
	In the stochastic setting, neither the proximal point nor exact function differences are available. APS therefore relies on two stochastic oracles: a \emph{stochastic difference oracle} (SDO) that returns a noisy estimate of $f(x)-f(y)$ for any queried pair $(x,y)$, and a \emph{stochastic proximal oracle} (SPO) that returns some candidate point given a center $x$ and a proximal parameter $\gamma$. We highlight one structural feature of the SPO. The SPO operates in two regimes: a \emph{safe regime} $\gamma\le 1/(2\rho)$, and the complementary regime $\gamma > 1/(2\rho)$. The accuracy requirement we impose on the SPO is deliberately lenient: we only ask that it be sufficiently accurate with some probability in the safe regime, and we impose no accuracy requirement whatsoever in the complementary regime. The two oracles are formally defined below.
	
	\paragraph{Stochastic Difference Oracle (SDO($\epsilon_f, q, \zeta_q$)).}
	Given a pair $(x,y)$,  with $x,y\in\dom f$, the oracle returns $\Delta(x,y,\Xi^\Delta(x,y))$, or $\Delta(x,y,\Xi^\Delta)$ for short, a random estimate of the difference $f(x)-f(y)$. The distribution of the random variable $\Xi^\Delta(x,y)$ may depend on $x$ and $y$. We allow the error in the estimate to be \emph{heavy-tailed}: rather than imposing a sub-Gaussian, sub-exponential, or bounded-noise assumption, we control only a finite $q$-th absolute central moment for some $q\ge 2$. Its accuracy is characterized through the absolute error
	\[
	\mathcal{E}(x,y,\Xi^\Delta):=\left|\Delta(x,y,\Xi^\Delta)-(f(x)-f(y))\right|,
	\]
	which we assume has bounded expectation and bounded $q$-th absolute central moment, for some $q\geq 2$:
	\begin{equation}\label{eq:zero_order_1}
		\mathbb{E}[\mathcal{E}(x,y,\Xi^\Delta)] \le \epsilon_f \quad \text{and} \quad \mathbb{E}\big[|\mathcal{E}(x,y,\Xi^\Delta)-\mathbb{E}[\mathcal{E}(x,y,\Xi^\Delta)]|^q\big] \le \zeta_q.
	\end{equation}
	The exponent $q$ governs how heavy the tails are allowed to be. In particular, $q=2$ reduces the assumption to bounded expectation and bounded variance. Larger $q$ corresponds to higher moments of the error being bounded, and sub-exponential and sub-Gaussian errors -- whose tails decay exponentially and admit bounded moments of all orders -- satisfy the assumption for every $q\ge 2$. 
	
	\begin{remark}[Two noisy function-value estimates as a special case]
		When only separate noisy function-value estimates $F_x \approx f(x)$ and $F_y \approx f(y)$ are available, the natural construction $\Delta(x,y) := F_x - F_y$ gives a
		difference estimate with error
		\[
		\left|\Delta(x,y) - (f(x)-f(y))\right|
		\le |F_x-f(x)| + |F_y-f(y)|;
		\]
		that is, the individual estimation errors add. Thus, whenever the right-hand side satisfies the required moment bounds, this construction fits the SDO model. The SDO
		formulation is more general, however, since it also admits coupled estimators, most notably common random numbers, where correlation between the two estimates can
		reduce the variance of the estimated difference.
	\end{remark}
	
	\paragraph{Stochastic Proximal Oracle (SPO($\nu(\cdot),\delta_p$)).}
	Given a point $x$ and a proximal parameter $\gamma > 0$, the oracle returns a candidate point $y(x, \gamma, \Xi^p(x,\gamma))$, or $y(x, \gamma, \Xi^p)$ for short, intended as a (random) approximate minimizer of the proximal subproblem
	\[
	\Phi(y) := f(y) + \frac{1}{2\gamma}\|y-x\|^2,
	\qquad
	\Phi^* := \min_{y} \Phi(y);
	\] 
	the minimum is attained for every $\gamma > 0$ but may be non-unique when $\gamma\rho \ge 1$. The distribution of the random variable $\Xi^p(x,\gamma)$ may depend on $x$ and $\gamma$. We impose a probabilistic accuracy requirement on the SPO output \emph{only in the safe regime} $\gamma \le \frac{1}{2\rho}$, where $\Phi$ is strongly convex and the proximal point $\prox_{\gamma f}(x)$ is uniquely defined. The requirement is that, with probability at least $1-\delta_p$ (conditional on the past), the optimality gap is bounded by an accuracy tolerance function $\nu(\gamma) > 0$:
	\begin{equation}\label{eq:prox_oracle}
		\mathbb{P}\!\bigl(\Phi(y(x,\gamma,\Xi^p)) - \Phi^* \le \nu(\gamma)\bigr) \ge 1 - \delta_p
		\qquad\text{for } \gamma \le \tfrac{1}{2\rho}.
	\end{equation}
	The function $\nu(\cdot)$ may depend on the parameter $\gamma$ and on other algorithm parameters; we will provide a concrete choice later. The pair $(\nu(\cdot), \delta_p)$, with $\delta_p \in (0,1)$, is intrinsic to the oracle. The oracle output is assumed to lie in $\dom f$ almost surely. In the safe regime $\gamma \le 1/(2\rho)$, $\Phi$ is $\mu$-strongly convex with $\mu = 1/\gamma - \rho \ge 1/(2\gamma)$, so the optimality-gap bound~\eqref{eq:prox_oracle} translates to a distance bound on the SPO output:
	\[
	\|y(x,\gamma,\Xi^p) - \prox_{\gamma f}(x)\| \;\le\; \sqrt{2\nu(\gamma)/\mu} \;\le\; 2\sqrt{\nu(\gamma)\,\gamma}
	\qquad\text{for } \gamma \le \tfrac{1}{2\rho}.
	\]
	
	\paragraph{Realizing the SPO.}
	The SPO can be realized by applying any first-order stochastic solver to the proximal subproblem $\Phi_k(y) = f(y) + \tfrac{1}{2\gamma_k}\|y-x_k\|^2$. On iterations with $\gamma_k\le\bar\gamma$, $\Phi_k$ is strongly convex; in this regime a standard stochastic subgradient solver attains the in-expectation rate $\mathbb E[\Phi_k(Y_k)-\Phi_k^*]=\mathcal O(1/T)$ that strong convexity affords, where $T$ is the number of steps the solver runs (e.g., \cite{Grimmer2019NonLipschitz} or \cite{LacosteJulienSchmidtBach2012}). Markov's inequality then converts this in-expectation guarantee into probabilistic accuracy guarantee required by the SPO. On iterations with $\gamma_k>\bar\gamma$, the same routine is run as a heuristic candidate generator; the returned point is treated only as a candidate for~\eqref{eq:unified_test}.
	
	Outside the safe regime ($\gamma > 1/(2\rho)$), no accuracy requirement is imposed on the SPO output: the returned point is treated only as a candidate, with no proximal-map or Moreau-gradient interpretation, and is accepted only when it passes the descent test.

	The deterministic setting of Section~\ref{sec:det_specialization} is the noiseless special case in which the SDO returns the exact difference $\Delta(x,y,\Xi^\Delta)\equiv f(x)-f(y)$ (so $\epsilon_f=\zeta_q=0$) and the SPO returns the exact proximal point on safe calls (so $\delta_p=0$ and $\nu(\cdot)\equiv 0$ inside the safe regime); on unsafe calls the SPO is again unconstrained.
	
	\paragraph{Stationarity measure.}
	Recall that we fix the reference parameter
	$
	\bar{\gamma} = \frac{1}{2\rho}.
	$ This is the standard scale at which stationarity is measured in the weakly-convex literature \cite{DavisGrimmerProximallyGuided,DDStochModelBased}: it is the largest proximal parameter at which the strong-convexity parameter of the proximal subproblem stays at least $\rho$ (leaving a safety margin from the degenerate boundary $\rho\gamma = 1$). The Moreau gradient at this parameter, $\nabla f_{\bar\gamma}(x) = \bar\gamma^{-1}(x - \prox_{\bar\gamma f}(x))$, is well defined and its norm serves as the way to measure stationarity throughout the analysis.
	
	\begin{definition}[Moreau-envelope stationarity]\label{def:eps_moreau_stat}
		A point $x \in \mathbb{R}^d$ is called \emph{$\varepsilon$-stationary in the Moreau-envelope sense} if $\|\nabla f_{\bar{\gamma}}(x)\| \le \varepsilon$.
	\end{definition}
	
	\begin{definition}[Stopping time, stochastic setting]\label{def:Teps_stoch}
		The stopping time associated with APS in this setting is
		\begin{equation}\label{eq:Teps_stoch}
			T_\varepsilon := \inf\bigl\{ k \ge 0 : \|\nabla f_{\bar\gamma}(x_k)\| \le \varepsilon\bigr\}.
		\end{equation}
	\end{definition}

	\subsection{Filtration and True Iteration}\label{sec:stoch_model}\label{sec:hp}

We use uppercase letters to denote random quantities, and all probabilities and expectations are taken with respect to the SPO and SDO randomness. Take $X_0=x_0$ and $\Gamma_0=\gamma_0$ to be deterministic. At iteration $k$, we define the following random variables:
	\begin{itemize}
		\item $X_k$: the iterate at the start of iteration $k$, with realization $x_k$;
		\item $\Gamma_k$: the proximal parameter at the start of iteration $k$, with realization $\gamma_k$;
		\item $Y_k$: the SPO output at iteration $k$, with realization $y_k$;
		\item $D_k := (X_k-Y_k)/\Gamma_k$: the Moreau-gradient proxy, with realization $d_k$;
		\item $\Delta_k := \Delta(X_k,Y_k,\Xi_k^\Delta)$: the SDO output at $(X_k,Y_k)$;
		\item $\mathcal E_k := \bigl|\Delta_k-\bigl(f(X_k)-f(Y_k)\bigr)\bigr|$: the absolute SDO error.
	\end{itemize}
	APS generates the stochastic process $\{(X_k,\Gamma_k,Y_k,D_k,\Delta_k,\mathcal E_k)\}$ adapted to the filtration $\{\mathcal F_k : k\ge -1\}$, where $\mathcal F_{-1}$ is trivial and $\mathcal F_k := \sigma(Y_0,\Delta_0,\ldots,Y_k,\Delta_k)$. All the random variables above are $\mathcal F_k$-measurable; the state $(X_k,\Gamma_k)$ is in addition $\mathcal F_{k-1}$-measurable. Within an iteration the SPO is queried before the SDO, so we will also use the intermediate $\sigma$-algebra
	\[
	\mathcal G_k := \mathcal F_{k-1}\vee\sigma(Y_k),
	\]
	which records the information available immediately after the SPO call and satisfies $\mathcal F_{k-1}\subseteq\mathcal G_k\subseteq\mathcal F_k$. The candidate $Y_k$ and $D_k$ are $\mathcal G_k$-measurable.

 Let $\Phi_k(y):=f(y)+\|y-X_k\|^2/(2\Gamma_k)$ and $\Phi_k^*:=\min_y\Phi_k(y)$ denote the random instance of the proximal subproblem~\eqref{eq:prox_oracle} from Section~\ref{sec:oracles}. We now introduce the definition of a true iteration, and an assumption related to it.

\begin{definition}[True iteration]\label{def:true_iter}
Let
$
\nu(\gamma) = \frac{\varepsilon^2\gamma}{32}.
$
Define the indicator random variable
\[
J_k :=
\begin{cases}
\mathbf 1\{\Phi_k(Y_k)-\Phi_k^*\le \nu(\Gamma_k)\}, & \Gamma_k\le \bar\gamma,\\[1mm]
1, & \Gamma_k>\bar\gamma.
\end{cases}
\]
We say iteration $k$ is \emph{true} if $J_k = 1$ and  the SDO error satisfies $\mathcal E_k\le \epsilon_f$.
We define the indicator
\[
I_k := \mathbf 1\{\text{iteration $k$ is true}\}.
\]
\end{definition}

\begin{assumption}[Probability of true iterations]\label{ass:true_iter_prob_hp}
There exists $p\in(0,1]$ such that, for every $k\ge0$,
\[
\mathbb P(I_k=1\mid \mathcal F_{k-1})\ge p
\qquad
\text{a.s. on } \{k<T_\varepsilon\}.
\]
\end{assumption}
One sufficient condition for \autoref{ass:true_iter_prob_hp} to hold is if $\mathbb P(\mathcal E_k\le\epsilon_f\mid\mathcal G_k)\ge 1-\delta_d$ and $\mathbb P(J_k=1\mid\mathcal F_{k-1})\ge 1-\delta_p$. In that case, \autoref{ass:true_iter_prob_hp} holds with $p = 1 - \delta_p - \delta_d$ by a union bound; no independence assumptions between SDO and SPO are needed.

\subsection{High-Probability Iteration Complexity}\label{sec:hp_thm}

To achieve the high-probability iteration complexity result, we will first show that safe true iterations necessarily pass the noisy descent test, and then use a counting argument to prove sufficiently many true iterations force either a small Moreau gradient or enough observable descent. For the analysis in the stochastic setting, we will take $\varepsilon_{\mathrm{rej}} = \frac{\varepsilon}{4}$.

\begin{lemma}[Safe true iterations are successful]\label{lem:st_implies_s}
Suppose Assumption~\ref{ass:weak_convexity} holds.
For any iteration $k<T_\varepsilon$, if $U_k = 0$ and $I_k=1$, then $\Theta_k=1$. Consequently,
\[
I_k(1-U_k)\le \Theta_k(1-U_k)
\qquad\text{for all }k<T_\varepsilon.
\]
\end{lemma}
\begin{proof}
Fix $k<T_\varepsilon$ with $U_k = 0$ and $I_k=1$. Since $U_k = 0$, we know $\Gamma_k \leq \bar{\gamma}$. Thus \(1-\rho\Gamma_k\in[1/2,1]\). Since $k<T_\varepsilon$,
\[
\|\nabla f_{\bar\gamma}(X_k)\|>\varepsilon.
\]
By Lemma~\ref{cor:moreau_factor2}, applied with $\gamma=\Gamma_k$ and $\lambda=\bar\gamma$,
\[
\|\nabla f_{\Gamma_k}(X_k)\|>
\frac{\varepsilon}{2(1-\rho\Gamma_k)}.
\]
Because $I_k=1$ implies $J_k=1$, the SPO error satisfies 
$$\Phi_k(Y_k)-\Phi_k^*\le \nu(\Gamma_k)
= \frac{\varepsilon^2\Gamma_k}{32},$$ so Lemma~\ref{lem:consistency} gives
\[
\|D_k-\nabla f_{\Gamma_k}(X_k)\|
\le
\sqrt{\frac{2\nu(\Gamma_k)}{\Gamma_k(1-\rho\Gamma_k)}}=\frac{\varepsilon}{4\sqrt{1-\rho\Gamma_k}}.
\]
Hence
\[
\begin{aligned}
\|D_k\| &\ge \|\nabla f_{\Gamma_k}(X_k)\|-\|D_k-\nabla f_{\Gamma_k}(X_k)\| \\
&> \varepsilon\left(\frac{1}{2(1-\rho\Gamma_k)}-\frac{1}{4\sqrt{1-\rho\Gamma_k}}\right)\ge \frac{\varepsilon}{4} = \varepsilon_{\mathrm{rej}},
\end{aligned}
\]
so the rejection safeguard is satisfied. For the middle inequality, set \(s=\sqrt{1-\rho\Gamma_k}\in[1/\sqrt 2,1]\); then
\((2-s)/(4s^2)\ge 1/4\).

It remains to verify the noisy descent inequality. Let $\bar Y_k:=\prox_{\Gamma_k f}(X_k)$ be the minimizer of $\Phi_k$. Since $\Gamma_k\le\bar\gamma$, $\Phi_k$ is strongly convex, and hence $\bar Y_k$ is unique. Since $J_k=1$,
\[
\Phi_k(Y_k)\le \Phi_k(\bar Y_k)+\nu(\Gamma_k),
\]
and therefore
\begin{equation}\label{eq:inexact_prox_descent_calc}
f(Y_k)
\le
f(\bar Y_k)
+\frac{1}{2\Gamma_k}\bigl(\|\bar Y_k-X_k\|^2-\|Y_k-X_k\|^2\bigr)
+\nu(\Gamma_k).
\end{equation}
Lemma~\ref{lem:accept}, applied to the exact proximal point $\bar Y_k$, gives
\[
f(X_k)-f(\bar Y_k)
\ge
\left(\frac{1}{\Gamma_k}-\frac{\rho}{2}\right)\|\bar Y_k-X_k\|^2.
\]
Combining this inequality with~\eqref{eq:inexact_prox_descent_calc} yields
\begin{align*}
f(X_k)-f(Y_k)-\frac{\sigma}{2}\Gamma_k\|D_k\|^2
&= f(X_k)-f(Y_k)-\frac{\sigma}{2\Gamma_k}\|Y_k-X_k\|^2 \\
&\ge
\frac{1-\rho\Gamma_k}{2\Gamma_k}\|\bar Y_k-X_k\|^2
+\frac{1-\sigma}{2\Gamma_k}\|Y_k-X_k\|^2
-\nu(\Gamma_k) \\
&\ge
\frac{\Gamma_k(1-\rho\Gamma_k)}{2}\|\nabla f_{\Gamma_k}(X_k)\|^2-\nu(\Gamma_k).
\end{align*}
Using $\|\nabla f_{\Gamma_k}(X_k)\|>\varepsilon/(2(1-\rho\Gamma_k))$ and the definition of $\nu(\Gamma_k)$,
\[
\frac{\Gamma_k(1-\rho\Gamma_k)}{2}\|\nabla f_{\Gamma_k}(X_k)\|^2-\nu(\Gamma_k)
>
\frac{\Gamma_k\varepsilon^2}{8(1-\rho\Gamma_k)}
-
\frac{\Gamma_k\varepsilon^2}{32}
\ge
\frac{3\Gamma_k\varepsilon^2}{32}
>0.
\]
The last inequality uses \(1-\rho\Gamma_k\le 1\).
Thus
\[
f(X_k)-f(Y_k)\ge \frac{\sigma}{2}\Gamma_k\|D_k\|^2.
\]
Finally, $I_k=1$ also gives $\mathcal E_k\le\epsilon_f$, so
\[
\Delta_k
\ge f(X_k)-f(Y_k)-\mathcal E_k
\ge \frac{\sigma}{2}\Gamma_k\|D_k\|^2-\epsilon_f.
\]
Both parts of the acceptance test hold, hence $\Theta_k=1$. 
\end{proof}

\begin{lemma}[Progress from successful steps]\label{lem:progress_bound} 
For every successful iteration $k$,
\begin{equation}\label{eq:any_succ_progress}
f(X_k)-f(X_{k+1})
\ge
\frac{\sigma}{2}\Gamma_k\varepsilon_{\mathrm{rej}}^2-\epsilon_f-\mathcal E_k.
\end{equation}
For every unsuccessful iteration $k$, $f(X_k)-f(X_{k+1})=0$.
\end{lemma}
\begin{proof}
If iteration $k$ is successful, then $X_{k+1}=Y_k$, so the acceptance test gives
\[
\Delta_k\ge \frac{\sigma}{2}\Gamma_k\|D_k\|^2-\epsilon_f,
\qquad
\|D_k\|\ge\varepsilon_{\mathrm{rej}},
\]
and the definition of $\mathcal E_k$ gives
\[
f(X_k)-f(Y_k)
\ge \Delta_k-\mathcal E_k
\ge \frac{\sigma}{2}\Gamma_k\varepsilon_{\mathrm{rej}}^2-
\epsilon_f-\mathcal E_k.
\]
If the iteration is unsuccessful, APS sets $X_{k+1}=X_k$.
\end{proof} 
 
Recall from \eqref{eq:h_des_def} that $
h_{\mathrm{des}} = 
\frac{\sigma\bar\gamma\varepsilon_{\mathrm{rej}}^2}{2\beta_{\mathrm{inc}}}.$
The proof below uses the log-scale distances \(\ell_k^+\) and \(\ell_k^-\) from~\eqref{eq:d0_def}. 

\begin{proposition}[Counting inequality]\label{prop:counting}
	Under Assumptions~\ref{ass:weak_convexity} and~\ref{ass:boundedness}, for any $t\le T_\varepsilon$,
	\begin{equation}\label{eq:counting}
	\sum_{k=0}^{t-1} I_k
	\le
	\frac{m}{h_{\mathrm{des}}}
	\left(
	Z_0+\sum_{k=0}^{t-1}\mathcal E_k
	\right)
	+
	\frac{1+m\epsilon_f/h_{\mathrm{des}}}{m+1}\,t  
	+
	\frac{
	m\ell_0^+
	+
	\left(1+m\epsilon_f/h_{\mathrm{des}}\right)\ell_0^-
	}{m+1}.
	\end{equation}
	\end{proposition}
	 
	\begin{proof}
		Note that if $\epsilon_f \geq h_{\mathrm{des}}$ then the RHS of \eqref{eq:counting} is larger than $t$, so the proposition trivially holds. Thus for the rest of the proof we assume that $\epsilon_f \leq h_{\mathrm{des}}$. 

		Lemma~\ref{lem:st_implies_s} gives $\sum_{k=0}^{t-1}I_k(1-U_k)\le\sum_{k=0}^{t-1}\Theta_k(1-U_k)$, hence
		\begin{equation}
			\label{eq:counting_true}
		\sum_{k=0}^{t-1}I_k
		\;\le\;
		\sum_{k=0}^{t-1}U_k
		+
		\sum_{k=0}^{t-1}\Theta_k(1-U_k)
		= \sum_{k=0}^{t-1}U_k + \sum_{k=0}^{t-1}\Theta_k
		- \sum_{k=0}^{t-1}U_k\Theta_k		\end{equation}
		The rest of the proof bounds the RHS of \eqref{eq:counting_true}.  To begin, define 
		\[
		r_k := \frac{\ln(\Gamma_k/\bar\gamma)}{\ln(1/\beta_{\mathrm{dec}})},
		\]
		so that a successful iteration has $r_{k+1}=r_k+m$ and an unsuccessful one has $r_{k+1}=r_k-1$. Define 
		\[
		\ell_k^+ := \max\{\lceil r_k\rceil,\,0\},
		\qquad
		\ell_k^- := \max\{\lfloor -r_k\rfloor,\,0\},
		\]
		which are both nonnegative integers. The changes in $\ell_k^+$ and $\ell_k^-$ across the four iteration types are summarized in \autoref{tab:reserve_changes} below; the reasoning is similar to the case analysis in the proof of \autoref{thm:exact_complexity}. 
		
		\begin{table}[ht] 
		\centering
		\begin{tabular}{|l|c|c|}
		\hline
		Iteration type & $\ell_{k+1}^+-\ell_k^+$ & $\ell_{k+1}^--\ell_k^-$ \\
		\hline
		Safe, successful   & $0$       & $-m$ \\
		Safe, unsuccessful & $0$       & $+1$ \\
		Unsafe, successful   & $\le +m$  & $\le 0$ \\
		Unsafe, unsuccessful & $-1$      & $0$ \\
		\hline
		\end{tabular}
		\caption{One-step changes of $\ell_k^+$ and $\ell_k^-$.}
		\label{tab:reserve_changes}
		\end{table}
		
		Adding up the changes in $\ell_k^+$ and $\ell_k^-$ over $k=0,\ldots,t-1$ using the bounds in Table~\ref{tab:reserve_changes},
		\[
		\begin{aligned}
		\ell_t^+
		&\le \ell_0^+ + m\sum_{k=0}^{t-1}U_k\Theta_k
		- \sum_{k=0}^{t-1}U_k(1-\Theta_k),\\
		\ell_t^-
		&\le \ell_0^- - m\sum_{k=0}^{t-1}\Theta_k(1-U_k)
		+ \sum_{k=0}^{t-1}(1-\Theta_k)(1-U_k).
		\end{aligned}
		\]
		Since $\ell_t^+\ge 0$, the first inequality implies
		\begin{equation}
			\label{eq:counting_1}
		\sum_{k=0}^{t-1}U_k  - (m+1)\sum_{k=0}^{t-1}U_k\Theta_k
		\;\le\;  \ell_0^+.
		\end{equation}
		Since $\ell_t^- \geq 0$, the second inequality implies 
		\begin{equation*}
		m\sum_{k=0}^{t-1}\Theta_k(1-U_k) - \sum_{k=0}^{t-1}(1-\Theta_k)(1-U_k)
		\;\le\;  \ell_0^-.
		\end{equation*}
		Because $\sum_{k=0}^{t-1}(1-\Theta_k)(1-U_k) = t - \sum_{k=0}^{t-1}U_k - \sum_{k=0}^{t-1}\Theta_k(1-U_k)$, the second inequality can be rewritten as
		\begin{equation}
			\label{eq:counting_2}
		(m+1)\sum_{k=0}^{t-1}\Theta_k(1-U_k) + \sum_{k=0}^{t-1}U_k 
		\;\le\;  t + \ell_0^-.
		\end{equation}
		Next, \autoref{lem:progress_bound} implies the following bound for any iteration $k$:
\[
f(X_k)-f(X_{k+1})
\ge
h_{\mathrm{des}}U_k\Theta_k-\Theta_k(\epsilon_f+\mathcal E_k).
\]
Telescoping and using $f(X_t)\ge f_{\inf}$, we get 
\[
h_{\mathrm{des}}\sum_{k=0}^{t-1}U_k\Theta_k
\le
Z_0+\sum_{k=0}^{t-1}\Theta_k(\epsilon_f+\mathcal E_k) \leq Z_0 + \sum_{k=0}^{t-1}\mathcal E_k + \epsilon_f \sum_{k=0}^{t-1}\Theta_k.
\]
	Thus
	\begin{equation}\label{eq:counting_3}
	h_{\mathrm{des}}\sum_{k=0}^{t-1}U_k\Theta_k-\epsilon_f \sum_{k=0}^{t-1}\Theta_k
	\le
	Z_0+\sum_{k=0}^{t-1}\mathcal E_k.
	\end{equation}
	We now combine~\eqref{eq:counting_1}, \eqref{eq:counting_2}, and \eqref{eq:counting_3} to bound the RHS of \eqref{eq:counting_true}. Since $\frac{\epsilon_f}{h_{\mathrm{des}}}\le1$, the coefficients
	\[
	\frac{m(1-\epsilon_f/h_{\mathrm{des}})}{m+1},\qquad
	\frac{1+m\epsilon_f/h_{\mathrm{des}}}{m+1}, \qquad 
	\frac{m}{h_{\mathrm{des}}}
	\] 
	are nonnegative. Multiplying
	\eqref{eq:counting_1}, \eqref{eq:counting_2}, and
	\eqref{eq:counting_3} by these three coefficients, respectively, and
	adding the resulting inequalities, we get
	\[
	\begin{aligned}
		\sum_{k=0}^{t-1}U_k + \sum_{k=0}^{t-1}\Theta_k - \sum_{k=0}^{t-1}U_k\Theta_k
		&\le
		\frac{m}{h_{\mathrm{des}}}\left(Z_0+\sum_{k=0}^{t-1}\mathcal E_k\right) \\
		&\quad
		+ \frac{m(1-\epsilon_f/h_{\mathrm{des}})}{m+1}\ell_0^+
		+ \frac{1+m\epsilon_f/h_{\mathrm{des}}}{m+1}(t+\ell_0^-).
	\end{aligned}
	\]
	Combining this with~\eqref{eq:counting_true} proves the proposition.
	\end{proof}

	Using this proposition, we can now prove the high-probability iteration complexity theorem.

	\begin{theorem}[High-probability iteration complexity]\label{thm:hp_main}
		Suppose Assumptions~\ref{ass:weak_convexity}, \ref{ass:boundedness}, and
		\ref{ass:true_iter_prob_hp} hold. Define
		\begin{equation}\label{eq:p0_R_def}
		p_0
		:=
		\frac{1}{m+1}
		+
		\frac{m(m+2)}{m+1}\frac{\epsilon_f}{h_{\mathrm{des}}},
		\qquad
		R
		:=
		\frac{mZ_0}{h_{\mathrm{des}}}
		+
		c_0,
		\end{equation}
		where $c_0 = \frac{m\ell_0^+ + (1+m\epsilon_f/h_{\mathrm{des}})\ell_0^-}{m+1}$. Suppose $p>p_0$, then for any $s\ge0$,  $
		\hat p\in
		\left(
		p_0+\frac{ms}{h_{\mathrm{des}}},\,p
		\right),$
		and any integer
		$
		t>
		\frac{R}{\hat p-p_0-ms/h_{\mathrm{des}}},
		$
		the stopping time satisfies
		\[
		\mathbb P(T_\varepsilon\le t)
		\ge
		1-
		\exp\!\left(-\frac{(p-\hat p)^2}{2}t\right)
		-
		\mathbb P\bigl(\overline{B_t(s)}\bigr),
		\]
		where
		$	
		B_t(s)
		:=
		\left\{
		\sum_{k=0}^{t-1}\mathcal E_k
		\le
		t(\epsilon_f+s)
		\right\}.
		$
		\end{theorem}
		
		\begin{proof}
		Let
		$
		A_t:=
		\left\{
		\sum_{k=0}^{t-1}I_k\ge \hat p t
		\right\}.
		$
		We have 
\begin{align*}
	\mathbb P(T_\varepsilon>t)
	&=
	\mathbb P(T_\varepsilon>t,\overline{A_t})+\mathbb P(T_\varepsilon>t, A_t)  \\
	&= 
	\mathbb P(T_\varepsilon>t,\overline{A_t})
	+ \mathbb P(T_\varepsilon>t, A_t, B_t(s))
	+ \mathbb P(T_\varepsilon>t, A_t, \overline{B_t(s)}) \\
	&\le
	\mathbb P(T_\varepsilon>t,\overline{A_t})
	+ \mathbb P(T_\varepsilon>t, A_t, B_t(s))
	+ \mathbb P\bigl(\overline{B_t(s)}\bigr).
\end{align*}
We first show that $\mathbb P(T_\varepsilon>t, A_t, B_t(s))=0$. On the event $\{T_\varepsilon>t, A_t, B_t(s)\}$, Proposition~\ref{prop:counting} gives
		\[
		\begin{aligned}
		\hat p t
		&\le
		\frac{m}{h_{\mathrm{des}}}
		\left(
		Z_0+\sum_{k=0}^{t-1}\mathcal E_k
		\right)
		+
		\frac{1+m\epsilon_f/h_{\mathrm{des}}}{m+1}t
		+
		\frac{
		m\ell_0^+
		+
		(1+m\epsilon_f/h_{\mathrm{des}})\ell_0^-
		}{m+1} \\
		&\le
		R
		+
		\left(
		\frac{1+m\epsilon_f/h_{\mathrm{des}}}{m+1}
		+
		\frac{m\epsilon_f}{h_{\mathrm{des}}}
		+
		\frac{ms}{h_{\mathrm{des}}}
		\right)t \\
		&=
		R+
		\left(
		p_0+\frac{ms}{h_{\mathrm{des}}}
		\right)t.
		\end{aligned}
		\]
		This contradicts the assumed lower bound on $t$. Hence $\mathbb P(T_\varepsilon>t, A_t, B_t(s))=0$, so
\[
\mathbb P(T_\varepsilon>t)
\le
\mathbb P(T_\varepsilon>t,\overline{A_t})+
\mathbb P\bigl(\overline{B_t(s)}\bigr).
\]
To bound the first probability, define
\[
M_t:=\sum_{k=0}^{t-1}(I_k-p)\mathbf 1_{\{k<T_\varepsilon\}}.
\]
Since $\mathbf 1_{\{k<T_\varepsilon\}}$ is $\mathcal F_{k-1}$-measurable and Assumption~\ref{ass:true_iter_prob_hp} gives $\mathbb E[I_k\mid\mathcal F_{k-1}]\ge p$ on $\{k<T_\varepsilon\}$, the process $(M_t)$ is a submartingale. Its increments are bounded by one in absolute value. On $\{T_\varepsilon>t,\overline{A_t}\}$,
\[
M_{t}
=
\sum_{k=0}^{t-1}(I_k-p)
< -(p-\hat p)t.
\]
Azuma--Hoeffding~\cite{Azuma1967} applied to the supermartingale $-M_t$ therefore yields
\[
\mathbb P(T_\varepsilon>t,\overline{A_t})
\le
\exp\!\left(-\frac{(p-\hat p)^2}{2}t\right).
\]
This proves the theorem.
		\end{proof}

\begin{remark}
Let's try to understand the condition $p>p_0$. First consider the noiseless-difference case $\epsilon_f=0$. In this case, the threshold reduces to
$
p>\frac{1}{m+1}.
$ When $\beta_{\mathrm{inc}}=\frac{1}{\beta_{\mathrm{dec}}}$, $m=1$, this condition is just $p>1/2$.
Since $m=\ln\beta_{\mathrm{inc}}/\ln(1/\beta_{\mathrm{dec}})$, the condition is equivalent to
\[
p\ln\beta_{\mathrm{inc}}-(1-p)\ln(1/\beta_{\mathrm{dec}})>0.
\]
Together with the fact that safe and true iterations are successful iterations, this implies that the proximal parameter has positive drift while the method is in the safe regime. This is a natural requirement for the algorithm to make progress, since it prevents the adaptive updates from being dominated by decreases, and allows the method to move toward a sufficiently large proximal parameter even when the initial proximal parameter is too small. 

The second term in $p_0$ is
$
\frac{m(m+2)}{m+1}\frac{\epsilon_f}{h_{\mathrm{des}}}.
$
This term compensates for the bias in the noisy function-difference estimates. The quantity $\epsilon_f/h_{\mathrm{des}}$ is the bias-to-progress ratio. This shows that given a fixed noise level $\epsilon_f$, the algorithm has a natural lower bound on the target accuracy $\varepsilon$ it can achieve, since $h_{\mathrm{des}}=\Theta(\varepsilon^2)$.
\end{remark}

The next corollary provides two standard ways to control the probability of the accumulated SDO error event $\overline{B_t(s)}$: a finite-moment model, which allows for heavy-tailed error distributions, and a sub-exponential distribution special case.

\begin{corollary}[SDO tail bounds]\label{cor:sdo}
Suppose Assumptions~\ref{ass:weak_convexity}, \ref{ass:boundedness}, and~\ref{ass:true_iter_prob_hp} hold, and $p>p_0$. Let
\[
\bar{\mathcal E}_k:=\mathbb E[\mathcal E_k\mid\mathcal G_k],
\qquad
W_k:=\mathcal E_k-\bar{\mathcal E}_k.
\]
Fix $s>0$ such that $p_0+ms/h_{\mathrm{des}}<p$, choose any $\hat p\in(p_0+ms/h_{\mathrm{des}},p)$, and let $t$ be any integer satisfying the lower bound in Theorem~\ref{thm:hp_main}.

\begin{enumerate}
\item If, for some $q\ge 2$ and $\zeta_q<\infty$, the SDO satisfies the $\mathcal G_k$-conditional moment bound
\begin{equation}\label{ass:sdo_conditional}
\bar{\mathcal E}_k \le \epsilon_f
\qquad\text{and}\qquad
\mathbb E[|W_k|^q\mid \mathcal G_k]\le \zeta_q
\qquad\text{almost surely,}
\end{equation}
then there exist constants $c_2>0$ and $C_q>0$, depending only on $q$, such that
\[
\mathbb P(T_\varepsilon>t)
\le
\exp\!\left(-\frac{(p-\hat p)^2}{2}t\right)
+
\exp\!\left(-\frac{c_2s^2}{\zeta_q^{2/q}}t\right)
+
\frac{C_q\zeta_q}{s^q t^{q-1}}.
\]

\item If, for some $\tau^2>0$ and $b>0$, the SDO satisfies the $\mathcal G_k$-conditional sub-exponential bound
\begin{equation}\label{ass:sdo_subexp}
\bar{\mathcal E}_k \le \epsilon_f,
\qquad
\mathbb E\!\left[\exp(\lambda W_k)\middle|\mathcal G_k\right]
\le
\exp(\lambda^2\tau^2/2)
\quad\text{for all } |\lambda|\le 1/b
\end{equation}
almost surely, then
\[
\mathbb P(T_\varepsilon>t)
\le
\exp\!\left(-\frac{(p-\hat p)^2}{2}t\right)
+
\exp\!\left(-\min\left\{\frac{s^2}{2\tau^2},\frac{s}{2b}\right\}t\right).
\]
\end{enumerate}
\end{corollary}
\begin{proof}
	Since $\mathcal F_{k-1}\subseteq\mathcal G_k\subseteq\mathcal F_k$ and
	$\mathbb E[W_k\mid\mathcal G_k]=0$, the tower property gives
	$\mathbb E[W_k\mid\mathcal F_{k-1}]=0$. Thus $(W_k)$ is an
	$(\mathcal F_k)$-martingale difference sequence. Moreover, in both cases
	$\bar{\mathcal E}_k\le \epsilon_f$, and hence
	$
	\overline{B_t(s)}
	\subseteq
	\left\{
	\sum_{k=0}^{t-1} W_k>st
	\right\}.
	$
	
	In the heavy-tailed case, the tower property and Jensen's inequality give
	$
	\mathbb E[|W_k|^q\mid\mathcal F_{k-1}]\le \zeta_q,
	\text{ and }
	\mathbb E[W_k^2\mid\mathcal F_{k-1}]
	\le \zeta_q^{2/q}.
	$
	Therefore the martingale Fuk--Nagaev inequality~\cite{FukNagaev1971,FanGramaLiu2015} yields
	\[
	\mathbb P(\overline{B_t(s)})
	\le
	\exp\!\left(-\frac{c_2s^2}{\zeta_q^{2/q}}t\right)
	+
	\frac{C_q\zeta_q}{s^q t^{q-1}}.
	\]
	In the sub-exponential case, the tower property gives
$
	\mathbb E\!\left[\exp(\lambda W_k)\middle|\mathcal F_{k-1}\right]
	\le
	\exp(\lambda^2\tau^2/2), \text{ for all }
	 |\lambda|\le 1/b,
	$
	so the martingale Bernstein inequality gives
	\[
	\mathbb P(\overline{B_t(s)})
	\le
	\exp\!\left(
	-\min\left\{\frac{s^2}{2\tau^2},\frac{s}{2b}\right\}t
	\right).
	\]
	Substituting the corresponding bound on $\mathbb P(\overline{B_t(s)})$
	into Theorem~\ref{thm:hp_main} proves the two estimates.
	\end{proof}

Corollary~\ref{cor:sdo} shows that when the SDO error follows a heavy-tailed distribution, the tail probability of the stopping time decays polynomially in $t$. In the special case of a sub-exponential error distribution, this decay is even faster, dropping exponentially in $t$. The following corollary shows that the stopping time is almost surely finite in both cases.
 
\begin{corollary}[Almost sure finite hitting time]\label{cor:as}
Under the conditions of Corollary~\ref{cor:sdo}, $\mathbb P(T_\varepsilon<\infty)=1$ in both cases.
\end{corollary}
\begin{proof}
The tail bound in Corollary~\ref{cor:sdo} tends to zero as $t\to\infty$, so
\[
\mathbb P(T_\varepsilon=\infty)
=
\lim_{t\to\infty}\mathbb P(T_\varepsilon>t)=0.
\]
\end{proof}

\section{Model-Based Extension of APS}\label{sec:model_oracles} 

The deterministic and stochastic analyses of Sections~\ref{sec:det_specialization}--\ref{sec:inexact_backtracking_proximal_point} use two key properties of the candidate oracle. First, in the safe regime, a candidate $y_k$ (with some probability) produces enough descent to pass the test~\eqref{eq:unified_test}. Second, in the safe regime, $d_k=(x_k-y_k)/\gamma_k$ is related to a stationarity measure of $f$: a small $\|d_k\|$ in a true iteration implies near-stationarity. 

This section shows that both properties are shared by an entire \emph{family} of candidate oracles, each obtained by minimizing a local model of $f$ plus a quadratic penalty. By changing the choice of local model, we obtain different algorithms: The proximal point algorithm is the case where the model is $f$ itself, and other choices of the model yield the prox-linear and proximal-gradient methods. This model-based view is based on \cite{DuchiRuan2018,DDStochModelBased}. The APS framework covers all such model-based algorithms, and the same analysis gives the same $\mathcal O(\varepsilon^{-2})$ iteration complexity for each. The rest of the section is organized as follows: we define the two-sided model and three specific examples, introduce the deterministic and stochastic model-based oracles built on it, and establish the resulting deterministic and stochastic high-probability complexity guarantees.  

\subsection{Two-Sided Models and Three Prox-Based Instances} 

\begin{definition}[Two-sided model]\label{def:model}
	Let \(x\in\dom f\) and let \(\tau,\eta\ge0\). A proper, lower semicontinuous function $m_x:\R^d\to\R\cup\{+\infty\}$ is a \emph{two-sided model of $f$ at $x$ with parameters $(\tau,\eta)$} if $m_x(x)=f(x)$, $\dom m_x=\dom f$, with:
\begin{align} \label{eq:two_sided_model}
  -\tfrac\tau2\|z-x\|^2\ \le\ f(z)-m_x(z)\ \le\ \tfrac\tau2\|z-x\|^2,\qquad z\in\dom f,
\end{align}
and $m_x$ is $\eta$-weakly convex with $\eta\le\tau+\rho$. 

We call $\gamma\le\bar\gamma_\tau:=\frac1{2(\tau+\rho)}$ the \emph{safe regime} (the range on which the model-based step could provably decrease $f$ and a small residual could imply near-stationarity (\autoref{lem:model_props})). The associated model subproblem is
\begin{equation}\label{eq:model_step}
  y^*=\argmin_y\left\{m_x(y)+\tfrac1{2\gamma}\|y-x\|^2\right\}.
\end{equation}
\end{definition}

Given such a model, the APS scheme of Algorithm~\ref{alg:aps} is unchanged: the only difference is that the candidate oracle now (approximately) solves the model subproblem~\eqref{eq:model_step} in place of the proximal subproblem. 

The power of the framework is that a single choice---the local model $m_x$---specializes \autoref{alg:aps} to a range of different methods, including the proximal-point, prox-linear, and proximal-gradient methods. We provide these three examples below. 

\paragraph{Proximal point ($\tau=0,\ \eta=\rho$).} Here $m_x=f$, so the optimal solution of the subproblem is $y^*=\prox_{\gamma f}(x)$ and the safe regime is $\gamma\le\bar\gamma=1/(2\rho)$. 

\paragraph{Prox-linear ($\tau=\rho,\ \eta=0$).} Suppose $f=h\circ c$ with $h$ convex and $L_h$-Lipschitz and $c$ of class $C^1$ with $L_c$-Lipschitz Jacobian, so that $\rho=L_hL_c$. Linearizing the inner map $c$ yields the model $m_x(z)=h\bigl(c(x)+\nabla c(x)(z-x)\bigr)$, which is convex and satisfies the model condition~\eqref{eq:two_sided_model} with $\tau=\rho$. In this case, the subproblem~\eqref{eq:model_step} is convex, and typically far easier to solve than the proximal subproblem. Sometimes it is solvable in closed form or in a few cheap iterations. The safe regime is $\gamma\le\bar\gamma_\rho=1/(4\rho)$. Two familiar cases are max-of-smooth, with $h(y) = \max(y_1, \ldots, y_m)$ and $c=(f_1,\dots,f_m)$, where $f_i$ are smooth functions, and {$\ell_1$-composite}, with $h=\|\cdot\|_1$ and $c$ a smooth residual map, so that $f(x)=\|c(x)\|_1$. The latter is the robust nonlinear-regression structure behind the phase-retrieval objective of \autoref{sec:numerics}, where $c_i(x)=(a_i^\top x)^2-b_i$. 

\paragraph{Proximal gradient ($\tau=L_g,\ \eta=0$).} For an additive composite $f=g+h$ with $g$ smooth ($L_g$-Lipschitz gradient, possibly nonconvex) and $h$ convex and prox-friendly, linearizing the smooth part yields the convex model $m_x(z)=g(x)+\langle\nabla g(x),z-x\rangle+h(z)$, which is convex and satisfies the model condition~\eqref{eq:two_sided_model} with $\tau=L_g$. 
Solving the subproblem~\eqref{eq:model_step} reduces to the proximal-gradient step $y^*=\prox_{\gamma h}\bigl(x-\gamma\nabla g(x)\bigr)$, and the safe regime is $\gamma\le\bar\gamma_{L_g}=1/(2(L_g+\rho))$. Familiar special cases include $h=0$ (gradient descent), $h=\iota_C$ for a closed convex set $C$ (projected gradient, $y^*=P_C(x-\gamma\nabla g(x))$), and $h=\|\cdot\|_1$ (soft-thresholding, an ISTA-type step); a difference of convex $f=q-s$ with $s$ smooth gives the proximal DC algorithm, with $g=-s$, so the concave part $-s$ is linearized.

These instances trade generality for per-step cost. The proximal point makes no structural assumption on $f$, but its subproblem, the prox of $f$ itself, can be harder to solve. Prox-linear and proximal-gradient instead exploit structure ($f=h\circ c$ or $f=g+h$) so that the model subproblem is convex and often has a closed-form solution, trading some generality for a far cheaper step. 
\autoref{alg:aps} captures all of them in a single framework, so one may use whichever oracle the problem at hand affords. Table~\ref{tab:instances} summarizes the corresponding examples.

\begin{table}[htbp]
    \centering
    \caption{Examples covered by the model-based APS framework.}
    \label{tab:instances}
    \begin{tabular}{@{}L{2.9cm}L{2.5cm}L{5.1cm}@{}}
    \hline
    instance & linearizes & representative examples \\
    \hline
    proximal point & nothing &  proximal point method \\[2pt]
    proximal-gradient $(g{+}h)$ & smooth $g$ & gradient descent; projected gradient; ISTA-type steps; proximal DC steps \\[2pt]
    prox-linear $(h\circ c)$ & inner map $c$ & finite-max problems; robust nonlinear regression; generalized Gauss--Newton steps \\
    \hline
    \end{tabular}
\end{table}

In this general model-based setting, the parameters $\rho$, $\eta$, and the model accuracy parameter $\tau$ determine both the strong convexity of the model subproblem and the safe threshold $\bar\gamma_\tau$ on which the model step will be shown to decrease $f$. The corresponding thresholds are:
\begin{center}
\begin{tabular}{@{}L{3.4cm}L{2.4cm}L{3.0cm}@{}}
\hline
 & $(\tau,\eta)$ & safe $\bar\gamma_\tau$ \\
\hline
proximal point & $(0,\rho)$ & $\tfrac1{2\rho}$ \\[2pt]
prox-linear $(h\circ c)$ & $(\rho,0)$ & $\tfrac1{4\rho}$ \\[2pt]
prox-gradient $(g+h)$ & $(L_g,0)$ & $\tfrac1{2(L_g+\rho)}$ \\[2pt]
general model $m_x$ & $(\tau,\eta)$ & $\tfrac1{2(\tau+\rho)}$ \\
\hline
\end{tabular}
\end{center}

\subsection{{Deterministic Model-Based Setting}}

\begin{definition}[Deterministic model-based oracle]\label{def:det_oracle}
Given a point $x$ and a parameter $\gamma$, the \emph{deterministic model-based oracle} returns, in the safe regime $\gamma\le\bar\gamma_\tau$, the minimizer $y^*$ of the model subproblem~\eqref{eq:model_step}; when $\gamma>\bar\gamma_\tau$ no accuracy is required of the output and the output need only lie in $\dom f$.
\end{definition}
This is the model-based analogue of the safe-exact oracle in \autoref{sec:det_specialization}, with the model $m_x$ in place of $f$ itself. The next lemma shows that the two important properties needed by the analysis hold in the safe regime: sufficient descent and a stationarity bound. 

\begin{lemma}[Deterministic descent and stationarity bounds]\label{lem:model_props}
Let $m_x$ be a two-sided model with parameters $(\tau,\eta)$, $\gamma\le\bar\gamma_\tau$, $x\in\dom f$, and $y^*$ minimize the model subproblem~\eqref{eq:model_step}, with $d=(x-y^*)/\gamma$. Then
\[
    f(x)-f(y^*)\ \ge\ \tfrac{2-(\tau+\eta)\gamma}{2}\,\gamma\|d\|^2\ \ge\ \tfrac\gamma2\|d\|^2,
    \qquad
    \|\nabla f_\gamma(x)\|\ \le\ 3\|d\|.
\]
\end{lemma}

\begin{proof}
Write $\Phi:=f+\tfrac1{2\gamma}\|\cdot-x\|^2$ and $M:=m_x+\tfrac1{2\gamma}\|\cdot-x\|^2$. On the safe regime both are strongly convex: $\Phi$ with parameter $\tfrac1\gamma-\rho>0$ and $M$ with parameter $\tfrac1\gamma-\eta>0$. So, each has a unique minimizer: $p:=\prox_{\gamma f}(x)$ for $\Phi$, and $y^*$ for $M$. In particular $\nabla f_\gamma(x)=(x-p)/\gamma$.

Since $M$ is $(\tfrac1\gamma-\eta)$-strongly convex on the safe regime with minimizer $y^*$, we have $M(x)\ge M(y^*)+\tfrac12(\tfrac1\gamma-\eta)\|x-y^*\|^2$. Combining $M(x)=m_x(x)=f(x)$ with the upper model bound $f(y^*)\le m_x(y^*)+\tfrac\tau2\|y^*-x\|^2=M(y^*)-\bigl(\tfrac1{2\gamma}-\tfrac\tau2\bigr)\|y^*-x\|^2$, we have 
\[
\begin{aligned}
  f(x)-f(y^*)
  &\ \ge\ \bigl[M(x)-M(y^*)\bigr]+\bigl(\tfrac1{2\gamma}-\tfrac\tau2\bigr)\|y^*-x\|^2\\
  &\ \ge\ \Bigl(\tfrac1\gamma-\tfrac{\tau+\eta}2\Bigr)\|y^*-x\|^2
  \ =\ \tfrac{2-(\tau+\eta)\gamma}{2}\gamma\|d\|^2\ \ge\ \tfrac\gamma2\|d\|^2.
\end{aligned}
\]
The last step uses $(\tau+\eta)\gamma\le1$ on the safe regime (since $\eta\le\tau+\rho$ and $\gamma\le\tfrac1{2(\tau+\rho)}$).

A $\mu$-strongly convex function exceeds its minimum value by at least $\tfrac\mu2\|\cdot-z^*\|^2$, where $z^*$ is its minimizer; applied to $\Phi$ at $p$ and to $M$ at $y^*$ this gives
\[
  \Phi(y^*)-\Phi(p)\ \ge\ \tfrac12\bigl(\tfrac1\gamma-\rho\bigr)\|y^*-p\|^2,
  \qquad
  M(p)-M(y^*)\ \ge\ \tfrac12\bigl(\tfrac1\gamma-\eta\bigr)\|y^*-p\|^2.
\]
Adding them, we obtain:
\begin{align*}
	\tfrac12\bigl(\tfrac2\gamma-\rho-\eta\bigr)\|y^*-p\|^2\ &\le [f(y^*)-m_x(y^*)]+[m_x(p)-f(p)]\\
	&\le \tfrac\tau2\bigl(\|y^*-x\|^2+\|p-x\|^2\bigr)
\end{align*}
Since $\eta\le\tau+\rho$ and $\gamma\le 1/(2(\tau+\rho))$, we have $\tfrac2\gamma-\rho-\eta\ge 3\tau+2\rho$. Hence
\begin{equation}
  \label{eq:det_model_lem_pf}
  \|y^*-p\|^2\le
  \frac{\tau}{3\tau+2\rho}\bigl(\|y^*-x\|^2+\|p-x\|^2\bigr)
  \le \tfrac13\bigl(\|y^*-x\|^2+\|p-x\|^2\bigr).
\end{equation}
Set $a:=\|\nabla f_\gamma(x)\|=\|p-x\|/\gamma$ and $b:=\|d\|=\|y^*-x\|/\gamma$. Using the reverse triangle inequality, we have
$$\abs{a-b} = \frac{1}{\gamma}\big|\norm{p-x} - \norm{y^* - x}\big| \leq \frac{1}{\gamma}\norm{(p-x) - (y^* - x)} = \frac{1}{\gamma}\norm{y^*-p}.$$
Plugging this back into \eqref{eq:det_model_lem_pf}, we get $(a-b)^2 \leq \frac13 (a^2 + b^2)$, i.e. $a^2 - 3ab + b^2 \leq 0$. This implies $a\le\frac{3+\sqrt{5}}{2}\,b<3b$. 
\end{proof}
  
\autoref{lem:model_props} immediately gives the $\mathcal{O}(\varepsilon^{-2})$ iteration complexity for model-based APS in the deterministic setting. 

\begin{theorem}[Deterministic model-based APS]\label{thm:det_model}
Let the candidate oracle be the deterministic model-based oracle (\autoref{def:det_oracle}) for a two-sided model with parameters $(\tau,\eta)$, and let function differences be exact. Run \autoref{alg:aps} with {$\sigma\in(0,1)$}, $\epsilon_f=0$, $\beta_{\mathrm{inc}}=1/\beta_{\mathrm{dec}}$, and $\varepsilon_{\mathrm{rej}}={\varepsilon/6}$. Then \autoref{alg:aps} produces, within
\[
  \mathcal O\bigl(\ell_{0,\tau}+(\tau+\rho)(f(x_0)-f_{\inf})/\varepsilon^2\bigr)
\]
iterations, an iterate $x_{k_*}$ with $\gamma_{k_*}\le\bar\gamma_\tau$ that is $\varepsilon$-stationary: $\|\nabla f_{\bar\gamma}(x_{k_*})\|\le\varepsilon$. Here, $\ell_{0,\tau}$ is the log-scale distance from $\gamma_0$ to $\bar\gamma_\tau$ (the analogue of~\eqref{eq:d0_def} with $\bar\gamma$ replaced by $\bar\gamma_\tau$). For the proximal point oracle ($\tau=0$), the stronger guarantee $\dist(0,\partial f(y_{k_*}))\le\varepsilon$ of \autoref{thm:exact_complexity} holds.
\end{theorem}
\begin{proof}
By \autoref{lem:model_props}, every safe iteration with $\|d_k\|\ge\varepsilon_{\mathrm{rej}}$ produces sufficient descent {($f(x_k)-f(y_k)\ge\tfrac12\gamma_k\|d_k\|^2\ge\tfrac\sigma2\gamma_k\|d_k\|^2$ for any $\sigma\in(0,1)$)} and is therefore successful. The counting argument of \autoref{thm:exact_complexity} then applies as before, with the distance $\ell_0$ replaced by $\ell_{0,\tau}$. This bounds the first index $k_*$ with $\gamma_{k_*}\le\bar\gamma_\tau$ and $\|d_{k_*}\|\le\varepsilon_{\mathrm{rej}}$. At $k_*$, the stationarity bound of \autoref{lem:model_props} gives
\[
  \|\nabla f_{\gamma_{k_*}}(x_{k_*})\|\le 3\|d_{k_*}\|\le 3\varepsilon_{\mathrm{rej}}={3\varepsilon/6=\varepsilon/2},
\]
and \autoref{cor:moreau_factor2}, applied with $\gamma=\gamma_{k_*}\le\bar\gamma$, converts this to the reference scale: $\|\nabla f_{\bar\gamma}(x_{k_*})\|\le 2\|\nabla f_{\gamma_{k_*}}(x_{k_*})\|\le\varepsilon$. When $\tau=0$, the identity $d_k\in\partial f(y_k)$ yields the subgradient guarantee directly.
\end{proof}

\subsection{Stochastic Model-Based Setting}

Replacing the exact minimizer of the deterministic oracle by an approximate one gives the \emph{stochastic} model-based oracle. We generalize the stochastic proximal oracle of \autoref{sec:oracles} from approximately solving the proximal subproblem to solving the model subproblem, keeping its defining feature: accuracy is asked for only in the safe regime, and only with constant probability.

\begin{definition}[Stochastic model-based oracle $\mathrm{SMO}(\nu_{\mathrm{m}}(\cdot),\delta_p)$]\label{def:smo}
Given $x$ and $\gamma>0$, the oracle returns a candidate $y(x,\gamma,\Xi^p)$, a random approximate minimizer of the model subproblem
\[
  M(y):=m_x(y)+\tfrac1{2\gamma}\|y-x\|^2.
\]
Just like the SPO, accuracy is required \emph{only in the safe regime}. In that regime, $M$ has the unique minimizer $y^*:=\argmin_y M(y)$, and with probability at least $1-\delta_p$ conditional on the past,
\begin{equation}\label{eq:smo}
  M\bigl(y(x,\gamma,\Xi^p)\bigr)-M(y^*)\le\nu_{\mathrm{m}}(\gamma)\qquad\text{for }\gamma\le\bar\gamma_\tau,
\end{equation}
and no requirement is imposed when $\gamma>\bar\gamma_\tau$. The pair $(\nu_{\mathrm{m}}(\cdot),\delta_p)$ is intrinsic to the oracle. The oracle output lies in $\dom f$ almost surely. Because $M$ is $(\tfrac1\gamma-\eta)$-strongly convex on the safe regime,~\eqref{eq:smo} yields a distance bound on the output:
\[
  \|y-y^*\|\le\sqrt{2\nu_{\mathrm{m}}(\gamma)\big/(\tfrac1\gamma-\eta)}.
\]
\end{definition}

The SMO generalizes the SPO. When $m_x=f$ (the proximal point, $\tau=0$, $\eta=\rho$), the model subproblem is the proximal subproblem and the SMO \emph{is} the SPO of \autoref{sec:oracles}. For the prox-linear and proximal-gradient models ($\eta=0$) it is a genuinely different oracle. 
 
On the good event where the accuracy (\ref{eq:smo}) is achieved, the descent and stationarity bound of \autoref{lem:model_props} still hold, degraded only by the oracle accuracy.

\begin{lemma}[Stochastic descent and stationarity bounds]\label{lem:smo_props}
Let $m_x$ be a two-sided model with parameters $(\tau,\eta)$, $\gamma\le\bar\gamma_\tau$, $x\in\dom f$, and let $y$ be an SMO output for which the good event~\eqref{eq:smo} holds, so that $M(y)\le M(y^*)+\nu_{\mathrm{m}}(\gamma)$. Writing $d=(x-y)/\gamma$,
\[
\begin{aligned}
    f(x)-f(y)&\ \ge\ \tfrac{2-(\tau+\eta)\gamma}{2}\gamma\|d\|^2-\sqrt{2\gamma\,\nu_{\mathrm{m}}(\gamma)}\,\|d\|-\nu_{\mathrm{m}}(\gamma),\\
    \|\nabla f_\gamma(x)\|&\ \le\ 3\bigl(\|d\|+\sqrt{2\nu_{\mathrm{m}}(\gamma)/(\gamma(1-\eta\gamma))}\bigr).
\end{aligned}
\]
\end{lemma}
\begin{proof}
For the descent bound, write $\mu_M:=\tfrac1\gamma-\eta>0$ as the strong-convexity parameter of $M$ on the safe regime. The upper model bound together with $M(x)=m_x(x)=f(x)$ gives 
\[
\begin{aligned}
  f(y)&\le m_x(y)+\tfrac\tau2\|y-x\|^2=M(y)-\bigl(\tfrac1{2\gamma}-\tfrac\tau2\bigr)\|y-x\|^2,\\
  f(x)-f(y)&\ \ge\ [M(x)-M(y)]+\bigl(\tfrac1{2\gamma}-\tfrac\tau2\bigr)\|y-x\|^2.
\end{aligned}
\]
Strong convexity at the minimizer $y^*$ gives $M(x)-M(y^*)\ge\tfrac{\mu_M}2\|x-y^*\|^2$, and the good event gives $M(y)-M(y^*)\le\nu_{\mathrm{m}}(\gamma)$, hence $\|y-y^*\|\le\sqrt{2\nu_{\mathrm{m}}(\gamma)/\mu_M}$. Thus $M(x)-M(y)\ge\tfrac{\mu_M}2\|x-y^*\|^2-\nu_{\mathrm{m}}(\gamma)$, and since $\|x-y^*\|\ge\|x-y\|-\|y-y^*\|$ gives $\|x-y^*\|^2\ge\|x-y\|^2-2\|x-y\|\,\|y-y^*\|$, 
\[
\begin{aligned}
  \tfrac{\mu_M}2\|x-y^*\|^2
  &\ \ge\ \tfrac{\mu_M}2\|x-y\|^2-\mu_M\|x-y\|\,\|y-y^*\|\\
  &\ =\ \tfrac{(1-\eta\gamma)\gamma}2\|d\|^2-(1-\eta\gamma)\,\|d\|\,\|y-y^*\|\\
  &\ \ge\ \tfrac{(1-\eta\gamma)\gamma}2\|d\|^2-\sqrt{2\gamma\,\nu_{\mathrm{m}}(\gamma)}\,\|d\|,
\end{aligned}
\]
where the equality uses $\mu_M\gamma=1-\eta\gamma$ (so $\tfrac{\mu_M}2\|x-y\|^2=\tfrac{(1-\eta\gamma)\gamma}2\|d\|^2$ and $\mu_M\|x-y\|=(1-\eta\gamma)\|d\|$), and the last inequality uses the distance bound $\|y-y^*\|\le\sqrt{2\nu_{\mathrm{m}}(\gamma)/\mu_M}=\sqrt{2\gamma\,\nu_{\mathrm{m}}(\gamma)/(1-\eta\gamma)}$, so $(1-\eta\gamma)\|y-y^*\|\le\sqrt{(1-\eta\gamma)\,2\gamma\,\nu_{\mathrm{m}}(\gamma)}\le\sqrt{2\gamma\,\nu_{\mathrm{m}}(\gamma)}$.
Combining the above, we have
\[
  f(x)-f(y)\ \ge\ \tfrac{2-(\tau+\eta)\gamma}{2}\gamma\|d\|^2-\sqrt{2\gamma\,\nu_{\mathrm{m}}(\gamma)}\,\|d\|-\nu_{\mathrm{m}}(\gamma).
\]

For stationarity, the distance bound from \autoref{def:smo} gives
\[
  \frac{\|y-y^*\|}{\gamma}
  \le
  \sqrt{\frac{2\nu_{\mathrm{m}}(\gamma)}{\gamma(1-\eta\gamma)}}.
\]
Using \autoref{lem:model_props} for the properties of the exact model minimizer $y^*$, and then the triangle inequality,
\[
  \|\nabla f_\gamma(x)\|
  \le
  3\frac{\|x-y^*\|}{\gamma}
  \le
  3\left(\|d\|+\frac{\|y-y^*\|}{\gamma}\right),
\]
which gives the claimed bound.
\end{proof}

\autoref{lem:smo_props} implies that the high probability $O(\varepsilon^{-2})$ iteration complexity of APS holds for two-sided models as well. For simplicity we state the result for convex models ($\eta=0$), which covers the prox-linear and proximal-gradient instances. The proximal point  ($\eta=\rho$) is already covered by \autoref{thm:hp_main} through the SPO. The argument extends to general $\eta$ upon replacing $\sqrt{2\nu_{\mathrm{m}}(\gamma)/\gamma}$ by $\sqrt{2\nu_{\mathrm{m}}(\gamma)/(\gamma(1-\eta\gamma))}$ in the proof and rescaling the accuracy constant.

\begin{theorem}[Stochastic model-based APS]\label{thm:stoch_model}
	Let the candidate oracle be an $\mathrm{SMO}$ (\autoref{def:smo}) for a
	convex model ($\eta=0$), and let the difference oracle be the SDO of
	\autoref{sec:oracles}. Run \autoref{alg:aps} with
	$\sigma\in(0,1)$ and $\varepsilon_{\mathrm{rej}}=\varepsilon/12$, and
	suppose that the SMO accuracy satisfies
	\[
	  \nu_{\mathrm{m}}(\gamma)
	  \le
	  \frac{1-\sigma}{6}\,
	  \gamma\,\varepsilon_{\mathrm{rej}}^2
	  \qquad\text{for }\gamma\le\bar\gamma_\tau.
	\]
	 
	Suppose Assumptions~\ref{ass:weak_convexity}--\ref{ass:boundedness}
	hold, and suppose the model-based analogue of
	\autoref{ass:true_iter_prob_hp} holds with $p>p_0$: namely, in the
	definition of a true iteration, replace the SPO accuracy event by the
	SMO good event~\eqref{eq:smo}, and define safe and unsafe iterations
	relative to $\bar\gamma_\tau$ rather than $\bar\gamma$.  Define $h_{\mathrm{des}}$, $\ell_0^\pm$, $p_0$, and
	$R$ in 
	\autoref{thm:hp_main} using $\bar\gamma_\tau$ in place of $\bar\gamma$. Then the conclusion of \autoref{thm:hp_main} holds. In particular, under the same SDO tail conditions as
	in \autoref{cor:sdo}, APS reaches the stopping time 
	\[
	  \|\nabla f_{\bar\gamma}(X_k)\|\le\varepsilon
	\]
	with high probability within $\mathcal O(\varepsilon^{-2})$
	iterations.
	\end{theorem}
	
	\begin{proof}
		It suffices to verify the model-based analogue of
		\autoref{lem:st_implies_s}, which is the only oracle-specific step
		needed in the proof of \autoref{thm:hp_main}: every
		model-based true iteration that is safe relative to
		$\bar\gamma_\tau$ and occurs before $T_\varepsilon$ is successful.
		
		Fix such an iteration $k<T_\varepsilon$. Then
		$\Gamma_k\le\bar\gamma_\tau\le\bar\gamma$, the SMO good
		event~\eqref{eq:smo} holds, and $\mathcal E_k\le\epsilon_f$. Since
		$k<T_\varepsilon$,
		\[
		  \|\nabla f_{\bar\gamma}(X_k)\|>\varepsilon.
		\]
		Hence, by \autoref{cor:moreau_factor2},
		\[
		  \|\nabla f_{\Gamma_k}(X_k)\|>\frac{\varepsilon}{2}.
		\]
		The accuracy condition implies
		\[
		  \sqrt{\frac{2\nu_{\mathrm{m}}(\Gamma_k)}{\Gamma_k}}
		  \le
		  \sqrt{\frac{1-\sigma}{3}}\,
		  \varepsilon_{\mathrm{rej}}
		  \le \varepsilon_{\mathrm{rej}}.
		\]
		Combining this with \autoref{lem:smo_props}, we get
		\[
		  \|D_k\|
		  >
		  \frac{\varepsilon}{6}
		  -
		  \sqrt{\frac{2\nu_{\mathrm{m}}(\Gamma_k)}{\Gamma_k}}
		  \ge
		  \frac{\varepsilon}{6}-\varepsilon_{\mathrm{rej}}
		  \ge
		  \varepsilon_{\mathrm{rej}},
		\]
		where the last inequality uses
		$\varepsilon_{\mathrm{rej}}=\varepsilon/12$. Thus the rejection
		safeguard is satisfied.
		
		Since $\eta=0$ and $\Gamma_k\le\bar\gamma_\tau$, we have $\tau\Gamma_k\le\tfrac12$, so $\tfrac{2-\tau\Gamma_k}{2}\ge\tfrac34$. Moreover, the descent bound of \autoref{lem:smo_props}, the inequality $\sqrt{2\Gamma_k\nu_{\mathrm{m}}(\Gamma_k)}\,\|D_k\|\le\tfrac14\Gamma_k\|D_k\|^2+2\nu_{\mathrm{m}}(\Gamma_k)$, and
		$\|D_k\|>\varepsilon_{\mathrm{rej}}$ give
		\begin{align*}
		  f(X_k)-f(Y_k)
		  &\ge
		  \tfrac34\Gamma_k\|D_k\|^2-\sqrt{2\Gamma_k\nu_{\mathrm{m}}(\Gamma_k)}\,\|D_k\|-\nu_{\mathrm{m}}(\Gamma_k)\\
		  &\ge
		  \tfrac12\Gamma_k\|D_k\|^2-3\nu_{\mathrm{m}}(\Gamma_k)\\
		  &>
		  \left(\tfrac12-\tfrac{1-\sigma}{2}\right)\Gamma_k\|D_k\|^2
		  \ =\ \frac{\sigma}{2}\Gamma_k\|D_k\|^2,
		\end{align*}
		using $\nu_{\mathrm{m}}(\Gamma_k)\le\tfrac{1-\sigma}{6}\Gamma_k\varepsilon_{\mathrm{rej}}^2<\tfrac{1-\sigma}{6}\Gamma_k\|D_k\|^2$ in the last inequality.
		Since $\mathcal E_k\le\epsilon_f$,
		\[
		  \Delta_k
		  \ge
		  f(X_k)-f(Y_k)-\mathcal E_k
		  \ge
		  \frac{\sigma}{2}\Gamma_k\|D_k\|^2-\epsilon_f.
		\]
		Both conditions of the acceptance test therefore hold, so iteration
		$k$ is successful.
		 
		This is the only oracle-specific step needed in the proof of
		\autoref{thm:hp_main}. Its counting and concentration arguments now
		apply similarly with safe
		and unsafe iterations defined relative to $\bar\gamma_\tau$ and $
		  h_{\mathrm{des}}
		  =
		  \frac{\sigma\bar\gamma_\tau\varepsilon_{\mathrm{rej}}^2}
			   {2\beta_{\mathrm{inc}}}.$
		The claimed conclusion follows.
		\end{proof}

The adaptive proximal-point, prox-linear, and proximal-gradient methods are thus three instances of one adaptive framework. The APS framework covers all of them, and the analysis applies to each. 

\section{Application to Expected Risk Objectives}\label{sec:sample_complex}

In this section, we apply APS with the stochastic oracle framework for stochastic objectives of the form
\[
f(x)=\mathbb E_\xi[F(x,\xi)],
\qquad \xi\sim P.
\] 
{\sloppy This class includes expected-risk minimization in machine learning, for a data point $\xi=d$ drawn from the data distribution $d \sim P$ and $F(x,d)=\ell(x;d)$, where $\ell$ is the loss function. The empirical-risk optimization problem is the finite-support specialization, with  $\xi\sim\mathrm{Unif}\{1,\ldots,n\}$, for which $f(x)=(1/n)\sum_{i=1}^n f_i(x)$.  It also covers simulation optimization, where $x$ is the decision variable and $\xi$ represents a random sample path.  We assume we have access to independent samples $F(\cdot,\xi)$ with $\xi\sim P$.\par}

Throughout this section, we assume that $f$ is $\rho$-weakly convex, finite-valued and bounded from below. We will show how the oracles from Sections~\ref{sec:inexact_backtracking_proximal_point} and \ref{sec:model_oracles} can be realized for this problem. The stochastic difference oracle is implemented by averaging paired differences $
F(x,\xi)-F(y,\xi),$ 
and the stochastic proximal or model-based oracles (SPO and SMO) are implemented by running the stochastic subgradient method on the proximal or model-based subproblem.

\subsection{Stochastic Difference Oracle}\label{sec:stoch_sdo}
 
Given a pair $(x,y)$, draw $\xi_1,\ldots,\xi_{b}$ independently from $P$ and return
\begin{equation}\label{eq:stoch_sdo}
\Delta(x,y):=\frac{1}{b}\sum_{j=1}^{b}\bigl(F(x,\xi_j)-F(y,\xi_j)\bigr).
\end{equation} 
This estimate uses common random numbers and can have much smaller variance than estimating \(f(x)\) and \(f(y)\) separately.
We assume that the centered paired difference has a uniformly bounded \(q\)th moment: for some \(q\ge2\) and \(\sigma_d<\infty\),
\begin{equation}\label{eq:diff_moments}
\mathbb E_\xi\!\left[
\left|F(x,\xi)-F(y,\xi)-\big(f(x)-f(y)\big)\right|^q
\right]\le \sigma_d^q,
\qquad \forall x,y.
\end{equation}
At iteration \(k\), when the oracle is queried at \((X_k,Y_k)\), write
\[
\mathcal E_k:=\left|\Delta(X_k,Y_k)-\big(f(X_k)-f(Y_k)\big)\right|,
\qquad
\bar{\mathcal E}_k:=\mathbb E[\mathcal E_k\mid\mathcal G_k],
\qquad
W_k:=\mathcal E_k-\bar{\mathcal E}_k.
\]

\begin{lemma}[Minibatch realizes the SDO model]\label{lem:sdo_minibatch}
Suppose~\eqref{eq:diff_moments} holds and the SDO is given by~\eqref{eq:stoch_sdo}. Then, for every iteration \(k\),
\[
\bar{\mathcal E}_k\le \frac{\sigma_d}{\sqrt{b}},
\qquad
\mathbb E\!\left[|W_k|^q\mid\mathcal G_k\right]\le \zeta_q := K_q\,\frac{\sigma_d^q}{b^{q/2}}
\]
for a constant $K_q$ depending only on $q$. 
\end{lemma}
\begin{proof}
Condition on \(\mathcal G_k\). Then \(X_k\) and \(Y_k\) are fixed, and the samples used by the SDO are independent of the past. Define
\[
\eta_j:=F(X_k,\xi_j)-F(Y_k,\xi_j)-\big(f(X_k)-f(Y_k)\big).
\]
Then \(\mathbb E[\eta_j\mid\mathcal G_k]=0\), and~\eqref{eq:diff_moments} gives
\(\mathbb E[|\eta_j|^q\mid\mathcal G_k]\le\sigma_d^q\). Since \(q\ge2\), Jensen's inequality also gives
\(\mathbb E[\eta_j^2\mid\mathcal G_k]\le\sigma_d^2\). Therefore
\[
\mathrm{Var}(\Delta(X_k,Y_k)\mid\mathcal G_k)\le \frac{\sigma_d^2}{b}.
\]
Jensen's inequality yields $\bar{\mathcal E}_k\le \sigma_d/\sqrt{b}$. The Marcinkiewicz--Zygmund inequality gives
\[
\mathbb E\!\left[\mathcal E_k^q\mid\mathcal G_k\right]
=
\mathbb E\!\left[\left|\frac{1}{b}\sum_{j=1}^{b}\eta_j\right|^q\middle|\mathcal G_k\right]
\le K_q\frac{\sigma_d^q}{b^{q/2}},
\]
for a constant \(K_q\) depending only on \(q\). Since \(W_k=\mathcal E_k-\bar{\mathcal E}_k\) and Jensen's inequality gives
\(\bar{\mathcal E}_k^q\le \mathbb E[\mathcal E_k^q\mid\mathcal G_k]\),
\[
\mathbb E[|W_k|^q\mid\mathcal G_k]
\le 2^{q-1}\left(
\mathbb E[\mathcal E_k^q\mid\mathcal G_k]+\bar{\mathcal E}_k^q
\right)
\le 2^qK_q\frac{\sigma_d^q}{b^{q/2}}.
\]
Enlarging \(K_q\) by a factor depending only on \(q\) proves the moment bound.
\end{proof}

\begin{corollary}[Minibatch size for small SDO error]\label{cor:sdo_minibatch_recipe}
	Fix $\delta_d\in(0,1)$ and $\epsilon_f>0$. If
	\begin{equation}\label{eq:stoch_b_choice}
	b\ge \frac{\sigma_d^2}{\delta_d\epsilon_f^2},
	\end{equation}
	then
	\[
	\mathbb P(\mathcal E_k\le\epsilon_f\mid\mathcal G_k)\ge 1-\delta_d,
	\qquad
	\bar{\mathcal E}_k\le\epsilon_f.
	\]
\end{corollary}
\begin{proof}
Chebyshev's inequality and the variance bound in the proof of Lemma~\ref{lem:sdo_minibatch} give
\[
\mathbb P(\mathcal E_k>\epsilon_f\mid\mathcal G_k)
\le \frac{\sigma_d^2}{b\epsilon_f^2}\le\delta_d,
\]
and Lemma~\ref{lem:sdo_minibatch} gives $\bar{\mathcal E}_k\le \sigma_d/\sqrt{b}\le\epsilon_f$.
\end{proof}

\begin{remark}
	The condition \(p>p_0\) in
	\autoref{thm:hp_main} requires
	\(\epsilon_f \leq c h_{\mathrm{des}}=\Theta(\varepsilon^2)\), with \(c>0\)
	sufficiently small. The minibatch condition~\eqref{eq:stoch_b_choice}
	therefore gives
$b=\Theta(\epsilon_f^{-2})=\Theta(\varepsilon^{-4})$ per SDO call. Since APS makes one SDO call per iteration and the iteration bound is \(\mathcal O(\varepsilon^{-2})\), the plain minibatch SDO construction uses \(\mathcal O(\varepsilon^{-6})\) samples over the entire run of the algorithm. This is the cost under the uniform noise assumption; other variance-reduced estimators may improve this bound when additional structure yields a smaller effective conditional variance, but such
refinements are outside the raw minibatch setting analyzed here.
\end{remark}

\subsection{Stochastic Proximal and Model-Based Oracles}\label{sec:stoch_spo}
We realize the SPO by running a stochastic subgradient method on the proximal subproblem $\Phi_k$ for at most $T_{\max}$ steps, using the subgradients discussed below. On a safe iteration where $\Gamma_k\le\bar\gamma$, the proximal subproblem is strongly convex with strong-convexity parameter at least $1/(2\Gamma_k)$, so one may use standard algorithms to find a point that approximately minimizes the strongly convex subproblem in reasonable time. On an unsafe iteration, the same routine may still be run, but its output is treated only as a candidate point and is accepted only if it passes the descent test.

For any $y$ and any $\xi$, we assume we can compute $g(y,\xi)$ satisfying
\[
g(y,\xi)\in\partial_y F(y,\xi),
\qquad
\mathbb E_\xi[g(y,\xi)]\in\partial f(y).
\]
Consequently, for any $x$ and parameter $\gamma>0$, a stochastic subgradient of
$
\Phi(y)=f(y)+\frac{1}{2\gamma}\|y-x\|^2
$
is
\[
g_\Phi(y,\xi):=g(y,\xi)+\gamma^{-1}(y-x).
\]
Since $\mathbb E_\xi[g(y,\xi)]\in\partial f(y)$, we have $\mathbb E_\xi[g_\Phi(y,\xi)]\in\partial\Phi(y)$.
Following the assumptions used in \cite{Grimmer2019NonLipschitz}, we assume there exist constants \(\sigma_0,\sigma_1,B_\Phi>0\) such that the following conditions hold on proximal subproblems where \(\gamma_k\le\bar\gamma\):
\begin{equation}\label{eq:sfo_growth}
\mathbb E_\xi\!\left[\|g_{\Phi_k}(y,\xi)\|^2\right]
\le
\sigma_0^2+\sigma_1\bigl(\Phi_k(y)-\Phi_k^*\bigr),
 \qquad 
\Phi_k(x_k)-\Phi_k^*\le B_\Phi.
\end{equation}
We emphasize that no conditions are required when $\gamma_k > \bar\gamma$. We now show that the stochastic subgradient method realizes the SPO.

\begin{lemma}[SGM realizes the SPO]\label{lem:spo_sgm}
Under Assumption~\ref{ass:weak_convexity} and condition~\eqref{eq:sfo_growth},
run the subgradient method of~\cite{Grimmer2019NonLipschitz} for $T_{\max}$ steps, using $\mu_k=(2\Gamma_k)^{-1}$ as its input strongly convex parameter.
On every iteration with $\Gamma_k\le\bar\gamma$, its output $Y_k$ satisfies
\begin{equation}\label{eq:spo_sgm_inexp}
\mathbb E\bigl[\Phi_k(Y_k)-\Phi_k^*\bigm|\mathcal F_{k-1}\bigr]
\le
\frac{8\sigma_0^2\Gamma_k}{T_{\max}} + \frac{8\sigma_1^2\bar\gamma B_\Phi\,\Gamma_k}{T_{\max}^2}
.
\end{equation}
Consequently, with $J_k$ as in Definition~\ref{def:true_iter}, the  SPO failure probability satisfies
\begin{equation}\label{eq:spo_sgm_failure_envelope}
\delta_p := \mathbb P(J_k=0\mid\mathcal F_{k-1})
\le 
{\frac{256}{\varepsilon^2T_{\max}}\left({\sigma_0^2} + \frac{\sigma_1^2\bar\gamma B_\Phi}{T_{\max}}\right)}
.
\end{equation}
\end{lemma}
 
\begin{proof}
On a safe iteration, $\Phi_k$ is strongly convex with parameter at least $\Gamma_k^{-1}-\rho\ge 1/(2\Gamma_k) = \mu_k$. The initial-gap condition and strong convexity give
$
\|X_k-\prox_{\Gamma_k f}(X_k)\|^2
\le
\frac{2\bigl(\Phi_k(X_k)-\Phi_k^*\bigr)}{\Gamma_k^{-1}-\rho}
\le
4\bar\gamma B_\Phi.
$
Running the stochastic subgradient method of \cite{Grimmer2019NonLipschitz} for $T_{\max}$ steps gives (see Theorem 7 in \cite{Grimmer2019NonLipschitz}):
\begin{equation*}
\mathbb E\bigl[\Phi_k(Y_k)-\Phi_k^*\bigm|\mathcal F_{k-1}\bigr]
\le
\frac{4\sigma_0^2}{\mu_k(T_{\max}+2)}
+
\frac{4\sigma_1^2\bar\gamma B_\Phi}{\mu_k(T_{\max}+1)(T_{\max}+2)}.
\end{equation*}
Since $\mu_k^{-1}=2\Gamma_k$, this is bounded by~\eqref{eq:spo_sgm_inexp}. Markov's inequality gives
\begin{align*}
\mathbb P(J_k=0\mid\mathcal F_{k-1})
=
\mathbb P\bigl(\Phi_k(Y_k)-\Phi_k^*>\nu(\Gamma_k)\mid\mathcal F_{k-1}\bigr) 
\le \frac{\mathbb E\bigl[\Phi_k(Y_k)-\Phi_k^*\bigm|\mathcal F_{k-1}\bigr]}{\nu(\Gamma_k)}.
\end{align*}
Substituting in \eqref{eq:spo_sgm_inexp} and the expression for $\nu(\Gamma_k)$ from \autoref{def:true_iter} proves~\eqref{eq:spo_sgm_failure_envelope}.
\end{proof}

\begin{remark}
	The condition~\eqref{eq:spo_sgm_failure_envelope} shows that, for any
	fixed SPO failure probability \(\delta_p\in(0,1)\), it is sufficient to take $T_{\max}=\Theta(\varepsilon^{-2}).$
	Thus, the stochastic proximal oracle requires \(\mathcal O(\varepsilon^{-2})\) stochastic subgradient evaluations per iteration. Combined with the \(\mathcal O(\varepsilon^{-2})\) iteration bound from \autoref{thm:hp_main}, this yields an overall sample complexity of \(\mathcal O(\varepsilon^{-4})\) stochastic subgradient evaluations over the entire run of the algorithm. 
\end{remark}

\paragraph{Realizing the stochastic model-based oracle.}
The SPO construction extends directly to the stochastic
model-based oracle of \autoref{def:smo}. For a convex model
($\eta=0$), apply the same stochastic subgradient method to 
\[
  M_k(y)
  =
  m_{X_k}(y)+\frac{1}{2\Gamma_k}\|y-X_k\|^2
\]
in place of $\Phi_k$. If $g_m(y,\xi)$ is an unbiased stochastic
subgradient of $m_{X_k}$, then
\[
  g_{M_k}(y,\xi)
  :=
  g_m(y,\xi)+\Gamma_k^{-1}(y-X_k)
\]
is an unbiased stochastic subgradient of $M_k$. Under the analogue
of~\eqref{eq:sfo_growth}, with $\Phi_k$ replaced by $M_k$ and $\bar{\gamma}$ replaced by $\bar\gamma_\tau$, the proof of \autoref{lem:spo_sgm}
applies.
Consequently the SMO accuracy required by \autoref{thm:stoch_model}
can be achieved with $T_{\max}=\mathcal O(\varepsilon^{-2})$ stochastic-subgradient steps, matching the sample complexity of the proximal oracle.

\section{Numerical Illustration}\label{sec:numerics}

We illustrate APS on robust phase retrieval, a standard nonsmooth weakly convex
test problem~\cite{DuchiRuanPhase,DuchiRuan2018,DDStochModelBased}:
\begin{equation*}
  f(x)=\frac{1}{n}\sum_{i=1}^{n}\bigl|\langle a_i,x\rangle^{2}-b_i\bigr|,
  \qquad a_i\sim\mathcal N(0,I_d).
\end{equation*}
We use $d=50$, $n=400$, and a planted signal $x_\star$ drawn uniformly from the
unit sphere ($x_\star=g/\|g\|$ with $g\sim\mathcal N(0,I_d)$). The clean measurements are
$b_i=\langle a_i,x_\star\rangle^2$, and we corrupt $5\%$ of them by replacing them with
large outliers $25\cdot\mathrm{Unif}(0,1)$.
The weak-convexity parameter of $f$ is
$
  \rho=\frac{2}{n}\lambda_{\max}(A^\top A),
  \bar\gamma=\frac{1}{2\rho}.
$
We use $\bar\gamma$ only to draw the fixed parameter baselines. APS
never uses $\rho$ or $\bar\gamma$.

{\sloppy
All runs start from the standard spectral initialization and report the
phase-retrieval error $\operatorname{dist}(x_k,\pm x_\star)$. The deterministic
experiments use a full-batch {\emph{high-accuracy}} proximal candidate, obtained by
running $10000$ subgradient steps from $x_k$ on
$\Phi_k(y)=f(y)+\frac{1}{2\gamma_k}\|y-x_k\|^2$ with stepsize $2\gamma_k/(t+2)$
and weighted averaging. The stochastic experiments replace this
with the stochastic proximal oracle (SPO) of \autoref{sec:stoch_spo}, using
$T=100$ minibatch subgradient steps, and use the stochastic difference oracle
(SDO) of \autoref{sec:stoch_sdo}. APS uses
$\sigma=\frac12$, $\beta_{\mathrm{inc}}=1/\beta_{\mathrm{dec}}=2$,
$\gamma_0=1$, $\varepsilon_{\mathrm{rej}}=10^{-10}$, and $\epsilon_f=0$.\par}

\paragraph{Comparison with fixed safe steps.}
\autoref{fig:aps_conv} compares APS with the fixed safe step $\bar\gamma=1/(2\rho)$
under the same budget for computing the proximal candidate. The fixed safe parameter baseline takes proximal steps with a fixed proximal parameter of $\bar\gamma$.
APS starts from the initial value $\gamma_0=1$ and uses no knowledge of $\rho$. It raises
$\gamma$ to take larger steps and makes more progress initially, and later makes smaller steps to keep converging  (here $\gamma$ climbs to ${\approx } 29\bar\gamma$ before decaying below $\bar\gamma$). As a result, the initial progress of APS is faster than that of the fixed safe parameter baseline, and its recovery error is smaller after $50$ iterations in the end.

\begin{figure}[tbp]
  \centering
  \includegraphics[width=0.6\linewidth]{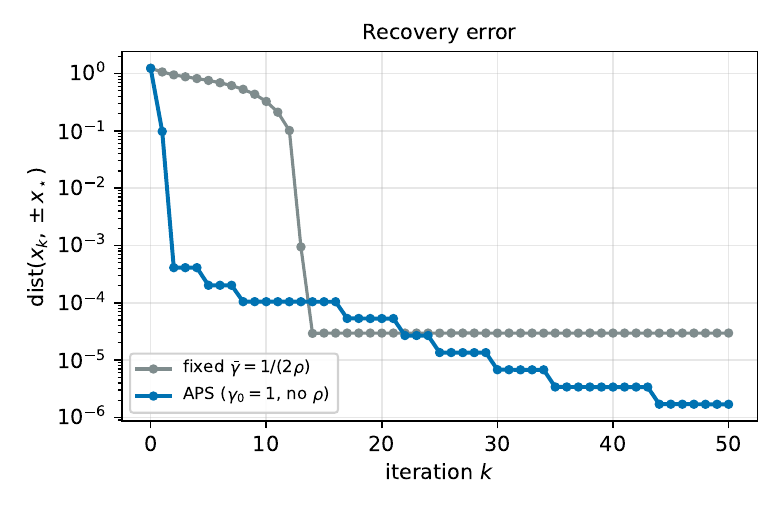}
  \caption{APS (which never uses $\rho$) versus the fixed safe step
  $\bar\gamma=1/(2\rho)$.}
  \label{fig:aps_conv}
\end{figure}

\paragraph{Sensitivity to the initial parameter.}
\autoref{fig:aps_stepsize} aggregates over $30$ random instances, in both the
deterministic and stochastic settings, the recovery error as a function of the initial proximal parameter $\gamma_0$. These curves fix the iteration budget at $K=50$, and both stochastic curves (fixed and APS) use minibatches of size $B=256$. 
For the method with a \emph{fixed} $\gamma$ the useful parameter range is narrow, since small parameter values make the progress slow and large parameter values make the progress unstable and can even cause divergence, with the
best region concentrated near the safe step $\bar\gamma$. The plot of APS, on the other hand, stays flat and low across the whole range, improving on the best fixed parameter without any tuning, in both deterministic and stochastic settings. 
 
\begin{figure}[tbp]
  \centering
  \includegraphics[width=0.66\linewidth]{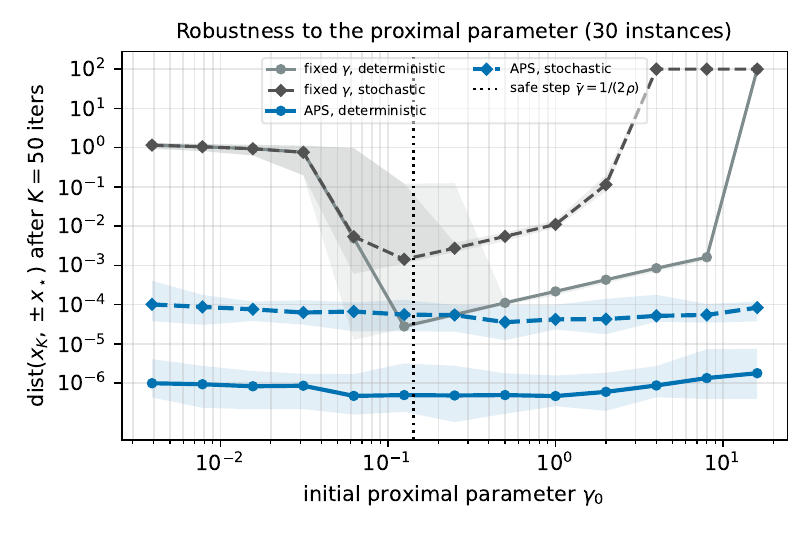}
  \caption{Robustness to the initial proximal parameter over $30$ random
  instances (medians after $K=50$ iterations; $10$--$90\%$ bands shaded for both methods).
 A \emph{fixed} $\gamma$ (gray) has a narrow useful window; APS (blue) stays flat and low throughout, showing its robustness to the initial choice of $\gamma_0$, without tuning or a
 weak-convexity estimate. The horizontal axis is the initial proximal parameter $\gamma_0$; the dotted line indicates the median safe $\bar\gamma$. The fixed parameter baselines have $\gamma=\gamma_0$ throughout, while APS starts with the initial value $\gamma_0$, and then adapts on its own.}
  \label{fig:aps_stepsize}
\end{figure}

\paragraph{Fully stochastic oracles.}
\autoref{fig:aps_stoch} illustrates the behavior under the stochastic setting.
It plots error against iteration at batch sizes $B\in\{64,128,256\}$, each run for
$K=50$ iterations. Each of the
$T=100$ inner subgradient steps of the SPO draws a fresh size-$B$ minibatch, and
the SDO draws another size-$B$ minibatch; thus APS costs
$(T{+}1)B$ samples per iteration versus $TB$ for the fixed parameter baseline, and larger
$B$ reduces the minibatch noise and reaches a lower error. The shaded areas are the $10$--$90\%$ spread over the $40$ runs, for both methods. Individual runs are non-monotone under the stochastic setting, but the medians move steadily. At every batch size the fixed safe parameter baseline plateaus, whereas APS adapts $\gamma$ downward as needed, and makes smaller steps to keep converging. APS reaches an error one to two orders of magnitude lower than that of the fixed safe parameter baseline. 
 
\begin{figure}[tbp]
  \centering
  \includegraphics[width=0.66\linewidth]{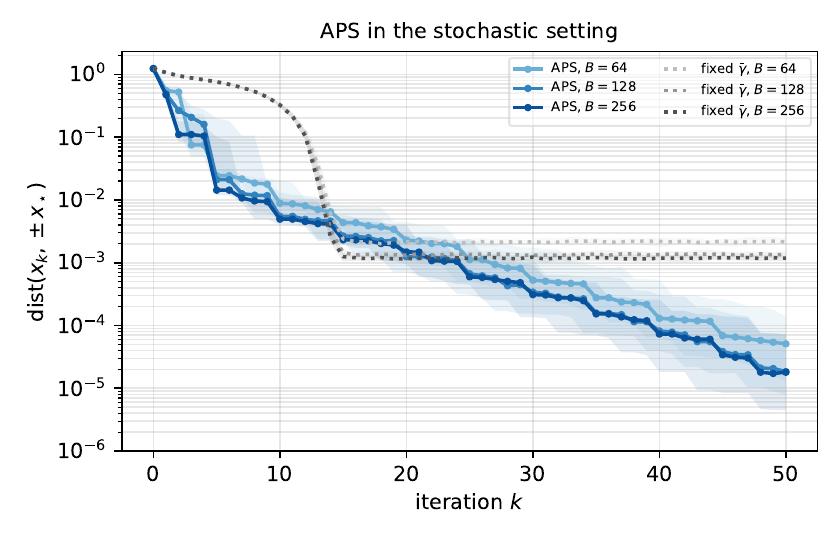}
  \caption{Fully stochastic setting (stochastic proximal and difference oracles), error versus iteration at batch sizes $B\in\{64,128,256\}$ (larger $B$ has less noise, $K=50$). APS (blue) adapts $\gamma$ down to keep converging and its recovery error stays one to two orders of magnitude below the fixed parameter baseline (gray), which plateaus.}
  \label{fig:aps_stoch}
\end{figure}
 
\FloatBarrier

\section{Conclusion}\label{sec:conclusion}

We introduced the Adaptive Prox-Guided Scheme (APS), an adaptive proximal framework for
weakly convex optimization with unknown weak-convexity parameter, requiring neither Lipschitz continuity nor smoothness of the objective. APS updates the proximal
parameter bidirectionally through a single accept/reject test, and never uses the safe threshold. The analysis splits
each run into safe and unsafe regimes, which are agnostic to the algorithm. We proved an $\mathcal O(\varepsilon^{-2})$ 
iteration complexity in the deterministic setting. In the stochastic setting, we proved a high-probability $\mathcal O(\varepsilon^{-2})$ iteration complexity under deliberately weak oracles, allowing biased, heavy-tailed function-difference
estimates and a candidate oracle that is reasonably accurate only with constant probability on safe
calls. We extended the analysis to the model-based setting, obtaining the same $\mathcal O(\varepsilon^{-2})$ iteration complexity in both the deterministic and stochastic regimes for the prox-linear, proximal-gradient, and any general model-based method falling under the framework.

\FloatBarrier

\bibliographystyle{plain}
\bibliography{references}

\end{document}